\documentclass[11pt]{amsart}

\usepackage{xypic}\xyoption{all}
\usepackage{amsthm,euscript}
\usepackage{paralist} \usepackage{amsopn}
\usepackage{amsfonts}
\usepackage{amsmath}
\usepackage{latexsym}
\usepackage{amscd}
\usepackage{amssymb}
\usepackage{amsmath}
\usepackage{fancyhdr}
\usepackage{setspace}
\usepackage{bm}%% for bold brackets
\usepackage[colorlinks=true, pdfstartview=FitV, linkcolor=blue,
citecolor=blue,
urlcolor=blue,breaklinks=true]{hyperref}
\newcommand{\sm}{\smallskip}
\newcommand{\ms}{\medskip}
\long\def\ignore#1\endignore{\relax}
\swapnumbers
\newcommand{\lv}[1]{} %% use this if the argument of \lv should not appear
\newcommand{\comments}[1]{}
\newcommand{\arxiv}[1]{\href{http://arxiv.org/abs/#1}{\tt arXiv:\nolinkurl{#1}}}
\newcounter{Cequ}
\newenvironment{CEquation}
  {\stepcounter{Cequ}%
     \addtocounter{equation}{-1}%
     \equation}
   {\endequation}
\newtheorem{theorem}[subsection]{Theorem}

\newtheorem{KMS}[subsection]{Kervaire-Milnor-Steinberg Theorem}

\theoremstyle{definition}

\numberwithin{equation}{subsection}
\newif \ifnew
\newtrue
\newcommand{\dbl}{(\!\!(}%old definition {(\!(}%double bracket left for commutator
\newcommand{\dbr}{)\!\!)}%old definition {)\!)}%double bracket right for commutator
\newcommand{\dd}[2]{\dbl #1, \, #2 \dbr}
\newcommand{\DBL}{\Big(\!\!\!\!\Big(}%old definition {(\!(}%double bracket left for commutator
\newcommand{\DBR}{\Big)\!\!\!\!\Big)}%old definition {)\!)}%double bracket right for commutator
\def\lan{\langle}
\def\Lan{\big\langle}
\def\ran{\rangle}
\def\Ran{\big\rangle}
\def\llan{\Lan \kern -2.5pt \lan}
\def\rran{\ran \kern -2.5pt \Ran}
\def\llc{\big[\kern -1.5pt [\,}
\def\rrc{\, ]\kern -1.5pt \big]}
\def\ch{\sp{\scriptscriptstyle\vee}}
\def\bu{\bullet}
\newcommand {\me}{^{-1}}
\newcommand{\ti}{^\times}
\def\x{{\rm x}}
\def\e{{\rm e}}
\def\xp{{\rm x}_+} \def\xpp{{\rm x}_+^\prime}
\def\xm{{\rm x}_-} \def\xmp{{\rm x}_-^\prime}
\def\xs{{\rm x}_\sigma}
\def\xms{{\rm x}_{-\sigma}}
\def\b{{\rm b}}
\def\rmw{{\rm w}}
\newcommand{\lra}{\longrightarrow}

\newcommand{\longto}{\longrightarrow}

\def\co{\colon}
\def\ot{\otimes}

\def\inpr{(\kern.2em|\kern.2em)}% inner product symbol
\def\edge{\mathrel{\vcenter{\hbox to 1.3em{\hrulefill}}}}
\def\lla{\leftarrow}
\def\lra{\rightarrow}
\def\dotcup{\mathbin{\dot{\cup}}}%  disjunkte Vereinigung
\def\alt{\sp{\rm alt}}
\def\her{\sp{\rm her}}
\def\qf{\sp{\rm qf}}

\def\vog{\dashv}
\def\coll{\mathrel{\top}}
\def\usi{_{\sigma}}
\def\umsi{_{ - \sigma}}
\newcommand{\sma}[4]{\big(\begin{smallmatrix} #1 & #2 \\ #3 &#4
              \end{smallmatrix}\big)} % small matrix
\newcommand{\pma}[4]{\begin{pmatrix} #1 & #2 \\ #3 &#4
              \end{pmatrix}} % matrix with parenthesis
\newcommand{\RS}{{\bf RS}}

\newcommand{\Aut}{\operatorname{Aut}}

\DeclareMathOperator{\Ad}{Ad} \DeclareMathOperator{\uAd}{\overline{\Ad}}
\DeclareMathOperator{\ad}{ad}

\newcommand{\Alt}{\operatorname{Alt}}

\newcommand{\can}{\operatorname{can}}
\newcommand{\diag}{\operatorname{diag}}

\newcommand{\End}{\operatorname{End}}

\newcommand{\EL}{\operatorname{E}}

\newcommand{\GL}{{\operatorname{GL}}}

\newcommand{\rmH}{\operatorname{H}}
\newcommand{\Id}{\operatorname{Id}}

\newcommand{\Ker}{\operatorname{Ker}}

\newcommand{\Mat}{{\operatorname{Mat}}}

\newcommand{\PE}{{\operatorname{PE}}}

\newcommand{\rank}{\operatorname{rank}}

\newcommand{\SL}{{\operatorname{SL}}}

\newcommand{\Span}{\operatorname{Span}}

\newcommand{\St}{\operatorname{St}}

\DeclareMathOperator{\rmZ}{Z} %for the centre
\newcommand\scC{\mathcal{C}}

\newcommand\scS{\mathcal{S}}

\newcommand\scZ{\mathcal{Z}}
\newcommand{\frA}{\mathfrak A}
\newcommand{\fra}{\mathfrak a}
\newcommand{\frb}{\mathfrak b}

\newcommand{\fre}{\mathfrak e}

\newcommand{\gl}{\mathfrak{gl}}

\newcommand{\tkk}{\mathfrak L}

\newcommand{\p}{\mathfrak p}

\newcommand{\R}{\mathfrak R}
\newcommand{\frS}{\mathfrak S}
\newcommand{\lsl}{\mathfrak{sl}}

\newcommand{\frz}{\mathfrak z}
\newcommand{\sfa}{{\sf a}}

\newcommand\sfm{{\sf m}}

 \newcommand\sfx{\sf x}
\newcommand\sfy{\sf y}

\newcommand{\bbA}{{\mathbb A}}

\newcommand{\FF}{{\mathbb F}}

\newcommand{\HH}{{\mathbb H}}
\newcommand{\JJ}{{\mathbb J}}
\newcommand{\MM}{{\mathbb M}}
\newcommand{\NN}{{\mathbb N}}

\newcommand{\QQ}{{\mathbb Q}}
\newcommand{\RR}{{\mathbb R}}
\newcommand{\ZZ}{{\mathbb Z}}
\newcommand\al{\alpha}\newcommand\alp{\al}
\newcommand\be{\beta}
\newcommand\ga{\gamma} \newcommand{\gam}{\ga}
 
\newcommand\de{\delta}\newcommand\del{\delta}
\newcommand\De{\Delta}\newcommand\Del{\Delta}
\newcommand\eps{\varepsilon} 
\newcommand\io{\iota}

 \newcommand\vphi{\varphi}

\newcommand\si{^\sigma} \newcommand\msi{^{-\sigma}}

\newcommand\sig{\sigma}

\def\rma{\dot {\rm A}}
\def\rmaf{{\rm A}}
\def\rmb{{\rm B}}
\def\rmc{{\rm C}}
\def\rmd{{\rm D}}
\def\rmbc{{\rm BC}}
\def\rme{{\rm E}}
\def\rmf{{\rm F}}
\def\rmg{{\rm G}}

\begin{document}

\title[]{Steinberg groups for Jordan pairs\\ -- an introduction with open problems}

\author{E. Neher}
\address{Department of Mathematics and Statistics, University of Ottawa,
    Ottawa, Ontario K1N 6N5, Canada}
\thanks{The author was partially supported by a Discovery grant from the Natural
Sciences and Engineering Research Council of Canada (NSERC)}
\email{neher@uottawa.ca}

\subjclass[2010]{}

\keywords{Jordan pairs, Steinberg groups}

\begin{abstract}  The paper gives an introduction to Steinberg groups for
root graded Jordan pairs, a theory developed in the book
\cite{LN} by Ottmar Loos and the author.
\end{abstract}

\maketitle

%\centerline{\em  Dedicated to Vyjayanthi Chari at the occasion of her 60th
%                          birthday }

\bigskip

\section*{Introduction}

The connection between Jordan structures (Jordan algebras, Jordan pairs) and
Lie algebras and groups has a long and successful history, starting with the
work of Chevalley-Schafer \cite{CS} and continued by Jacobson
\cite{Ja1,Ja2,Ja3}, Kantor \cite{Ka1,kan:trans}, Koecher
\cite{koe:imbedI,Ko2,Ko3,Ko69}, Loos \cite{hav,ag,simp}, Springer \cite{Sp},
Springer-Veldkamp \cite{SV} and Tits \cite{Ti0,tits:alt}.

The book \cite{LN} by Loos and the author is a further contribution to the
theme ``Groups and Jordan Structures''. It contains a detailed study of
Steinberg groups associated with certain types of Jordan pairs.
These groups generalize the classical linear and unitary Steinberg groups of a ring
by, roughly speaking, replacing associative
coordinates with Jordan algebras or Jordan pairs. We are able to
prove the basic results on Steinberg groups (central closedness, universal central extension in the stable case) in our setting, thereby recovering all previous results, except those on  groups of type $\rme_8$, $\rmf_4$ and $\rmg_2$, and in addition deal with new types, not considered before. The main novelty however is our approach based on $3$-graded root systems and Jordan pairs. \ms

The present paper is an introduction to the theory developed in \cite{LN}. In
\S\ref{sec:elg} we describe the linear Steinberg group $\St(A)$ of a ring $A$
from the point of view of Jordan pairs. This is motivation for \S\ref{sec:gen}
where we define the Steinberg group of a root graded Jordan pair and state the
main results of \cite{LN} regarding these groups. The final section
\S\ref{sec:open} discusses some open research problems in the area of Steinberg
groups and Jordan pairs.

The paper does not assume any prior knowledge of linear Steinberg groups or
Jordan pairs: all relevant definitions are given in the paper. We demonstrate
their scope by many examples and refer the reader to \cite{lfrs} and \cite{LN} for most proofs. But we include the details of our discussion of the linear Steinberg group and the elementary
linear group of a ring from the point of view of Jordan theory (\ref{ano} --
\ref{news} and \ref{exc} respectively). We also give all details of our
description of the Tits-Kantor-Koecher algebra and the projective elementary
group of a rectangular Jordan pair (\ref{rect-pvex}, \ref{pvn}). \sm

{\em Notation.} Throughout $k$ is a unital commutative associative ring and $A$
is a not necessarily commutative, but unital associative $k$-algebra. Its
identity element and zero element are denoted $1_A$ and $0_A$
respectively. We will often simply write $1$ for $1_A$ if $A$ is clear from the
context, and analogously for $0\in A$. We use $A\ti$ to denote the invertible
elements of $A$. If $k=\ZZ$ we will refer to $A$ as a ring.

For non-empty sets $I$ and $J$ we denote by $\Mat_{IJ}(A)$ the $k$-module of
$I\times J$-matrices over $A$, i.e., maps $x: I \times J \to A$ with only finitely many values different from $0$.  As usual, we write a matrix in the form $x =
(x_{ij})_{(i,j) \in I\times J}$. In case $I=J$ we
abbreviate $\Mat_I(A) = \Mat_{IJ}(A)$. This is an associative $k$-algebra with
respect to ordinary matrix multiplication which is unital if
and only if $I$ is finite.
We put $\Mat_n(A)= \Mat_I(A)$ if $|I|=n< \infty$.
Here and in general $|I|$ denotes the cardinality of the set $I$. The identity
element of $\Mat_n(A)$ is denoted ${\bf 1}_n$, and the group
$\Mat_n(A)\ti$ by $\GL_n(A)$.

%For a $k$-module $M$ we denote by $\End_k(M)$ its endomorphism algebra and by
%$\gl(M)$ its underlying Lie algebra, whose Lie algebra product $[\cdot, \cdot
%]$ is given by commutators.

The group commutator of elements $g,h$ in a group $G$ is $\dbl g,h\dbr = ghg\me
h\me$. \ms

%{\em Dedication.} This paper is dedicated to Vyjayanthi Chari in honour of
%her many important contributions to mathematics. \sm

{\em Acknowledgement.}  The author thanks Ottmar Loos for many helpful comments on an earlier version of this paper.

\section{Elementary linear groups and their Steinberg groups}\label{sec:elg}

In this section we give an introduction to elementary linear groups over a ring
$A$ (\ref{elg}) and their associated Steinberg groups (\ref{stgr}). After a
review of central extensions in \ref{centex} we state the
Kervaire-Milnor-Steinberg Theorem (\ref{kms})
which says, for example, that the stable Steinberg group is the universal central extension of the stable elementary group.
We also exhibit a new set of generators and relations for the Steinberg groups considered in this section (\ref{ano} -- \ref{rw}), which we take as axioms for a new Steinberg group defined in \ref{news}. The main result is Theorem~\ref{mth}: the classical and the new Steinberg groups are isomorphic. \ms

\subsection{Elementary linear groups} \label{elg} Let $n\in \NN$, $n\ge 2$. As usual, $E_{ij}\in \Mat_n(A)$ is the
$n\times n$-matrix with entry $1_A$ at the position $(ij)$ and $0_A$ elsewhere.
For $1\le i\ne j\le n$ and $a\in A$ we put
\[
 \e_{ij}(a) = {\bf 1}_n + a E_{ij}, \quad (a\in A)
\]
The well-known multiplication rules $aE_{ij} \, bE_{kl} = \de_{jk}\, ab E_{il}$
for $a,b\in A$ imply
\begin{CEquation}\label{elg1}
\e_{ij} (a) \, \e_{ij}(b) = \e_{ij}(a + b).
\end{CEquation}
\lv{%%%%%%%%%%%%%%%%%%%%%%
$\e_{ij}(a) \e_{ij}(b) = (E_n + a E_{ij})(E_n + b E_{ij}) = E_n + a E_{ij} + b
E_{ij} + (aE_{ij})( bE_{ij}) = E_n + (a + b) E_{ij}$ since $i\ne j$ }%%%%%%%%%%%%%%%%
Hence $\e_{ij}(a) \, \e_{ij}(-a) = {\bf 1}_n = \e_{ij}(-a) \, \e_{ij}(a)$,
which shows that $\e_{ij}(a) \in \GL_n(A)$.
The {\em elementary linear group (of rank $n$)} is the subgroup
\[ \EL_n(A) = \Lan \e_{ij}(a) : 1\le i \ne j \le n, a\in A\Ran
\]
of $\Mat_n(A)\ti$ generated by all $\e_{ij}(a)$.  \ms

One easily verifies two further relations of the $\e_{ij}(a)$:
%  using the
% commutator $\dbl \cdot, \cdot \dbr$ of the group $\Mat_n(A)\ti$:
\begin{align}
\label{elg2} \dbl \e_{ij}(a), \, \e_{kl}(b) \dbr &= {\bf 1}_n \qquad \qquad
            (j\ne k, i \ne l), \tag{E2} \\
\label{elg3} \dbl \e_{ij}(a), \, \e_{jl}(b) \dbr &= \e_{il}(ab) \qquad (i,j,l\ne).
\tag{E3}
\end{align}
Taking the inverse of \eqref{elg3} and using $\dbl g,h\dbr \me = \dbl h,g\dbr$
yields the equivalent relation
\begin{equation} \label{elg4} \dbl \e_{ij}(a), \, \e_{ki}(b) \dbr = \e_{kj}(-ba)
\qquad (i,j,k\ne). \tag{E4}
\end{equation}
\lv{%%%%%%%%%%%%
Proof: $\dbl \e_{ij}(a), \, \e_{ki}(b) \dbr = \dbl \e_{ki}(b),
\e_{ij}(a)\dbr\me =
\e_{kj}(ba) \me = \e_{kj}(-ba)$}%%%%%%%%%%%%%%%%%%%%%%%%%
\sm

We will also need an infinite variant of $\Mat_n(A)$ and the group
$\EL_n(A)$. Let
\[ \Mat_\NN (A) \]
be the set of all $\NN \times \NN$-matrices $x=(x_{ij})_{i,j\in \NN}$
with entries from $A$. 
Recall that only finitely many $x_{ij} \ne 0$. The usual
addition and multiplication of matrices are well-defined operations on
$\Mat_\NN(A)$ satisfying all axioms of a ring, except the existence of an
identity element. To remedy this, let ${\bf 1}_\NN = \diag(1_A, 1_A , \ldots )$
be the diagonal matrix of size $\NN \times \NN$ with every diagonal entry being
$1_A$. Then
\[ \Mat_\NN(A)_{\rm ex} := k {\bf 1}_\NN + \Mat_\NN(A)\]
is a ring with the usual addition and multiplication of matrices. Its identity
element is ${\bf 1}_\NN$ and its zero element is the zero matrix, see for
example \cite[1.2B]{hahn} where this ring is denoted $\Mat_\infty(A)$
(its elements are the $\NN \times \NN$-matrices with entries from $A$ which have
only finitely many non-zero entries off the diagonal and whose diagonal
elements become eventually constant).

We associate with $x\in \Mat_n(A)$ the matrix $\io_n(x) \in
\Mat_\NN(A)_{\rm ex}$ by putting $x$ in the upper left corner and
filling the diagonal outside $x$ with $1_A$:
\[
\io_n(x) = \begin{pmatrix}
  x & 0 \\ 0 & \diag(1_A, \ldots) \end{pmatrix}
\]
Then $\io_n$ maps invertible matrices in $\Mat_n(A)$ to invertible matrices of
$\Mat_\NN(A)_{\rm ex}$, in particular $\io_n\big(\e_{ij}(a)\big)\in
\Mat_\NN(A)_{\rm ex}\ti$. Since $\io_n\big( \e_{ij}(a)\big) = \io_p\big(
\e_{ij}(a)\big)$ for $p \ge n$, we can take the maps $\io_n$ as identification
and view all $\e_{ij}(a)$, $i, j\in \NN$ with $i\ne j$, as elements of
$\Mat_\NN(A)\ti_{\rm ex}$. The {\em (stable) elementary linear group \/} is the
subgroup $\EL(A)$ of $\Mat_\NN(A)_{\rm ex}\ti$ generated by all the
$\e_{ij}(a)$:
\[
 \EL(A) = \Lan \e_{ij}(a) : i,j\in \NN, i \ne j \Ran.
\]
It is immediate that the relations \eqref{elg1} -- \eqref{elg4} also hold in
$\EL(A)$. The group $\EL(A)$ is canonically isomorphic to the limit of the
inductive system $(\EL_n(A), \io_{pn})$ where $\io_{pn} \co \EL_n(A) \to
\EL_p(A)$ for $p \ge n$ is defined by taking the left upper $(p \times
p)$-corner of $\io_n(x)$.

\subsection{Why is $\EL_n(A)$ important?} \label{eqrem}
One reason is that $\EL_n(F) = \SL_n(F)$ in case of $A=F$ is a field -- in other words, every matrix of determinant $1$ can be reduced to the identity matrix
by elementary row and column reductions. The equality $\EL_n(A) =
\SL_n(A)$ even holds for a noncommutative local ring, for example a division
ring, 
if one uses the Dieudonn\'e determinant (\cite[2.2.2]{hahn} or
\cite[2.2.2--2.2.6]{Ro}). If $A$ is commutative then obviously $\EL_n(A)
\subset \SL_n(A)$. Equality holds for example if $A$ is a Euclidean ring
\cite[1.2.11]{hahn}.

While all of this is interesting, the real interest in $\EL_n(A)$ and $\EL(A)$
stems from their connection to Steinberg groups of $A$ and to the K-group $K_2(A)$, defined in \eqref{stgrk2}.

\subsection{The Steinberg groups $\St_n(A)$ and $\St(A)$}\label{stgr} We assume $n\in \NN$, $n \ge 3$ (the case $n=2$ is uninteresting since then the definitions below yield free products of $A$. The group $\St_2(A)$ is defined differently, see e.g.\ \cite{milnor}; it will not play a role in this paper).
\sm

%By definition,
The {\em Steinberg group $\St_n(A)$\/} is the group presented by
\begin{itemize}
  \item[$\bullet$] generators $\x_{ij}(a)$, $1 \le i\ne j\le n$ and $a\in A$,
      and
  \item relations \eqref{elg1} -- \eqref{elg3} of \S\ref{elg}:
\begin{align*}
       \x_{ij}(a) \, \x_{ij}(b) &= \x_{ij}(a +
       b) \quad \hbox{ for all $1\le i\ne j \le n$ and $a,b\in A$}, \\
        \dbl \x_{ij}(a), \x_{kl}(b) \dbr &=1 \quad
            \hbox{if $j \ne k$ and $l \ne i$},\\
     \dbl \x_{ij}(a), \x_{jl}(b) \dbr &= \x_{il}(ab) \quad
            \hbox{if $i, j, l \ne$}.
\end{align*} \end{itemize}

The {\em (stable) Steinberg group $\St(A)$\/} is the group
presented by
\begin{itemize}
  \item generators $\x_{ij}(a)$, $i,j\in \NN$ distinct, $a\in A$, and

  \item relations \eqref{elg1} -- \eqref{elg3} for $i,j\in \NN$.
\end{itemize} \sm

Since the defining relations \eqref{elg1} -- \eqref{elg3} hold in $\EL_n(A)$
and $\EL(A)$ we get surjective group homomorphisms
\begin{equation} \label{stgwp}
  \wp_n \co \St_n(A) \to \EL_n(A)\qquad \hbox{and} \qquad \wp \co \St(A) \to \EL(A)
\end{equation}
determined by  $\x_{ij}(a) \mapsto \e_{ij}(a)$. The second K-group of $A$ is
then defined as
\begin{equation} K_2(A) := \Ker(\wp). \label{stgrk2} \end{equation}
This is an important but also mysterious group, even for fields. The reader
can find more about this group in the classic \cite{milnor}, and in \cite[Ch.~1]{hahn}, \cite[Parts IV and V]{magurn}, \cite[Ch.~4]{Ro}, or   \cite[III]{weibel} (the list is incomplete).
\sm

To put all of
this in a bigger picture, we make a short detour on central extensions of groups.
\lv{%%%%%%%%%%%%%%%%%%%%%%%%%%%%%%%%%%%%%%%%%%%%%%%%%%%%%%%%%%%%%%%%%%%%%
Outlook on $K_2(F)$, $F$ field:
\begin{enumerate}
  \item Matsumoto's Theorem \cite[4.3.15]{Ro}

  \item Relation to $_n {\rm Br}(F)$, \cite[4.4.18]{Ro} or Magurn, 16D

  \item $K_2(\QQ)$ \cite[4.4.9]{Ro}

  \item $|K_2(F)| = 1$ for a finite field \cite[4.3.13]{Ro}.

   \item $K_2(F)$, $F$ algebraically closed, is a divisible group [Lam --
       Intro to quadratic forms, Ex. 6.15]

   \item $K_2(F)$ is uncountable for an uncountable field (same reference
       as (5)
\end{enumerate}
One also knows $K_2(\ZZ)| = 2$ \cite[Remark p. 93, proven on p. 97]{st2}}%%%%%%%%%%%

\subsection{Central extensions}\label{centex} Let $G$ be a group.
An {\em extension of $G$\/} is a surjective group homomorphism $p \co E \to
G$. An extension is called {\em central\/} if $\Ker(p)$ is contained in the centre of the group $G$. A central extension $q\co X \to G$ is a {\em universal central extension\/} if for all central extensions $p \co E \to G$ there exists a unique homomorphism $f \co X \to E$ such that $q = p \circ f$:
 \[  \xymatrix{ X  \ar@{.>}[rr]^{\exists! \, f} \ar[dr]_q && E \ar[dl]^p \\ &G}
\]
A group $X$ is called {\em centrally closed\/} if $\Id_X \co X \to X$ is a
universal  central extension. Thus, $X$ is centrally closed if and only if
every central extension $p \co E \to X$ splits uniquely in the sense that
there exists a unique group homomorphism $f \co X \to E$ satisfying $p \circ
f = \Id_X$. The concepts defined above are related by the following facts,
proved for example in \cite[1.4C]{hahn}, \cite[\S5]{milnor}, \cite[4.1]{Ro}
or \cite[\S7]{st2}. \ms

\begin{enumerate}[(a)]
  \item\label{centex-a}  A group $G$ has a universal central extension if and
      only if it is {\em perfect\/}, i.e.,
      generated by all commutator $\dbl g, h\dbr$ of $g,h\in G$. In
      particular, a centrally closed group is perfect.\sm

\item For two universal central extensions of a group $G$, say $q \co X \to
    G$ and $q' \co X' \to G$, there exists an isomorphism $f \co X \to X'$ of
    groups, uniquely determined by the condition $q = q' \circ f$. \sm

\item \label{centex-b} Let $q \co X \to G$ be a universal central
    extension, whence $G$ is perfect by \eqref{centex-a}. Then $X$ is
    centrally closed and thus also perfect, again  by \eqref{centex-a}.
    Obviously, $G$ is a central quotient of $X$. The following fact
    \eqref{centex-c} says that every universal central extension of $G$ is
    obtained as a central quotient of a centrally closed group.
     \sm

\item \label{centex-c} A surjective group homomorphism $q \co X \to G$ is a
    universal central extension if and only if (i) $X$ is centrally closed
    and (ii) $\Ker(q)$ is central.  \sm

\item \label{centex-d} Let $f\co X \to G$ and $g\co G \to \bar G$ be
    central extensions. Then $f$ is a universal central extension if and
    only if $g \circ f $ is a universal central extension.
% \ref{centex-d} is Exercise (vi) of \cite[p. 76]{St2} (Chapter VII)
\end{enumerate}
\sm

 \noindent To describe a universal central extension of a group $G$ we have,
by \eqref{centex-c} and \eqref{centex-d}, two approaches:

\begin{enumerate}[(I)]

\item  \label{centexI}
Find successive central extensions $G_1 \to G_0=G, \ldots , G_n \to G_{n-1}, \ldots$ until one of them, say $G_n \to G_{n-1}$, becomes universal, and then take the composition $G_n \to G$ of these central extensions, or

\item \label{centexII} find an extension $q \co X \to G$ with $X$ centrally
    closed and then find conditions for $\Ker(q)$ to be central.
    \end{enumerate}

\noindent
Although \eqref{centexI} seems to be the more natural approach, 
we will actually follow \eqref{centexII}. But first back to Steinberg groups.
\sm

In \cite{st1} Steinberg gave a very elegant description of the universal central extension of a perfect Chevalley group over a field. ``Most''
Chevalley groups are perfect by \cite[Lemma~32]{st2}. In particular,
for $n\ge 2$ and $F$ a field, the group $\SL_n(F)= \EL_n(F)$ (equality by
\ref{eqrem}) is a Chevalley group, and it is perfect if $n \ge 4$ or if $n=3$
and $|F|\ge 3$ or if $n=2$ and $|F|\ge 4$. A special case of Steinberg's
result in \cite{st1} is the following theorem.

\begin{theorem}[{\cite{st1}, \cite[Thm.~10]{st2}, \cite[Thm.~1.1]{st3}}]\label{st-thm}
Let $n\in \NN$, $n\ge 2$ and let $F$ be a field satisfying $|F|> 4$ if $n=3$
and $|F| \not\in \{2,3,4,9\}$ if $n=2$. Then the map $\wp_n \co \St_n(F) \to
\EL_n(F)$ of \eqref{stgwp} is a universal central extension.
\end{theorem}
\ms

We have defined the maps $\wp_n\co \St_n(A) \to \EL_n(A)$ and $\wp\co \St(A)
\to \EL(A)$ for any ring $A$. It is therefore natural to ask if
Theorem~\ref{st-thm} even holds for rings. An answer is given in the
following Kervaire-Milnor-Steinberg Theorem.

\begin{KMS}[{\cite{kerv,milnor,st2}}]\label{kms} For an arbitrary ring $A$, \sm

\begin{inparaenum}[\rm (a)]
 \item\label{kms-a} the group $\St_n(A)$, $n \ge 5$, is centrally closed. \sm

\item \label{kms-b} The map $\wp \co \St(A) \to \EL(A)$ is a universal
    central extension.
\end{inparaenum}
\end{KMS}

An indication of the proof of \eqref{kms-a} can be found in see \cite[1.4.12]{hahn} or \cite[Cor.~1]{st2}. The attribution of part \eqref{kms-b} of this theorem is
somewhat complicated. Kervaire cites a preliminary version of \cite{milnor},
Milnor attributes it to Steinberg and Kervaire (\cite[p.~43]{milnor}), and
Steinberg says that \eqref{kms-b}, proved
in \cite[Thm.~14]{st2}, is ``based in part on a letter from J.~Milnor''. \sm

In view of \eqref{kms-a}, the map $\wp_n \co \St_n(A) \to \EL_n(A)$, $n\ge
5$, is a universal central extension if and only if it is a central
extension. It is known that this is not always the case, see
\cite[4.2.20]{hahn}. However, by \ref{centex}\eqref{centex-b}, both
$\St_n(A)$ and $\St(A)$ are centrally closed. It is this result that we will
be concentrating on, following the strategy \ref{centex}\eqref{centexII}.

\subsection{Another look at $\St_n(A)$ and $\St(A)$: using root systems.}\label{ano}
To treat $\St_n(A)$, $n\in \NN$, $n\ge 3$, and $\St(A)$ at the same time we use the subset $N\subset \NN$, which is the finite integer interval $N=[1,n]$ or $N=\NN$.
We can then put
\[
 \St_N(A) = \begin{cases}
   \St_n(A) & \hbox{if } N = [1, n], \\ \St(A) & \hbox{if } N = \NN.
 \end{cases}
\]
By definition in \ref{stgr}, $\St_N(A)$ is generated by $\x_{ij}(a)$, $(i,j)
\in N \times N$,
$i \ne j$, and $a\in A$. We replace this indexing set by the
root system $\rma_N$ (notation of \ref{class}),
which we realize in the Euclidean space $X=\bigoplus_{i\in N} \RR \eps_i$ with basis $(\eps_i)_{i\in N}$ and inner
product $\inpr$ defined by $(\eps_i | \eps_j) = \del_{ij}$:
\[
 R = \rma_N = \{\eps_i - \eps_j : i,j\in N\}, \qquad R\ti =  R \setminus \{0\}.
\]
Thus $R\ti=\rmaf_{n-1}$ for $N=[1,n]$ in the traditional notation, while for $N=\NN$ we get an infinite locally finite root system -- a concept that we will
review later in \ref{lfrs}. For the moment, it suffices to use the concretely
given $R$ above. \sm

For $\alp$, $\beta \in R$ one easily checks that $(\alp | \beta) \in \{0, \pm
1, \pm 2\}$ with $(\alp | \beta ) = \pm 2 \iff  \alp = \pm \beta$. To conveniently
describe the remaining cases we use the symbols
\begin{equation} \label{anopdef}
 \alp \perp \beta \iff (\alp | \beta) = 0 \qquad \hbox{and} \qquad \alp
\edge \beta \iff (\alp | \beta) = 1.
\end{equation}
A straightforward analysis of the indices shows for $\alp = \eps_i - \eps_j$
and $\beta = \eps_k - \eps_l \in R$ that
\begin{align}
\alp \perp \beta \; \hbox{ or } \; \alp \edge \beta \quad & \iff \quad j \ne k
    \; \hbox{ and }\;  l \ne i, \nonumber \\
\alp \edge (-\beta) \quad & \iff \quad j=k, i,j,l \ne \; \hbox{ or } \; i=l, i,j,k \ne.
\label{ano0} \end{align}
\lv{%%%%%%%%%%%%%%%%%%%%%%%%%%%%%%%%%%% lv starts
We will prove the two equivalences simultaneously. For $\alp$ and $\beta$ as
given we have \[ (\alp | \beta ) = \del_{ik} - \del_{il} - \del_{jk} +
\del_{jl} \] with $(\alp|\beta) = \pm 2 \iff \alp = \pm \beta \iff i=k, j=l
\hbox{ or } i=l, j=k$. It therefore suffices to show
\begin{align*}
  (\alp | \beta) \in \{0,1\} \quad &\implies \quad j \ne k \hbox{ and } l \ne i, \\
  (\alp | \beta) = -1 \quad &\implies \quad i,j=k,l \ne \hbox{ or } i=l,j,k \ne.
\end{align*}

Case $(\alp | \beta) = 1$: Then $\del_{ik} = 1$ or $\del_{jk} = 1$ but not
both, which means $i=k$, $j\ne l$ or $j=l$, $ i \ne k$. If $i=k$, $j \ne l$
then $j\ne k$ because $i \ne j$ and $i \ne l$ because $k \ne l$. If $j=l$,
$i\ne k$ then $j \ne k$ because $j=l \ne k$ and $i \ne l$ because $i \ne j =
l$.

Case $(\alp | \beta) = 0$: Then $\del_{ik} = 0 = \del_{il} = \del_{jk} =
\del_{jl}$ so in particular $\del_{jk} = 0 = \del_{il}$.

Case $(\alp | \beta) = -1$: Then $\del_{jk} = 1$ or $\del_{il} = 1$ but not
both, whence $j=k$, $i \ne l$ (then $i,j,l \ne$ because $i \ne j$) or $i=l$,
$j \ne k$ (then $i,j,k\ne$).}%%%%%%%%%%%%%%%%%  lv ends

\noindent Hence, putting
\[ \x_\al(a) = \x_{ij}(a) \qquad \hbox{for $\alp = \eps_i - \eps_j\in R\ti$},\]
the relations \eqref{elg1} -- \eqref{elg4} can be rewritten in terms of roots
as follows.  Let $\alp$, $\beta\in R\ti$ and denote the relations corresponding to (Ei) by (ERi).  Then the previous relations read
\begin{align}
\label{ano1} \dbl \x_\al(a), \, \x_\al(b) \dbr &= 1,
            \tag{ER1}\\
\label{ano2} \dbl \x_\al(a), \, \x_\beta(b) \dbr &= 1,
           & &\hbox{if $\alp \perp \beta$  or $ \alp \edge \beta$}, \tag{ER2} \\
\label{ano3} \dbl \x_\alp(a), \, \x_\beta (b) \dbr &= \x_{\alp + \beta}(ab) \qquad
  & &\hbox{if $\alp = \eps_i - \eps_j$, $\beta = \eps_j - \eps_l$, $i,j,l\ne$},
\tag{ER3} \\
\label{ano4} \dbl \x_\alp(a), \, \x_\beta (b) \dbr &= \x_{\alp + \beta}(-ba) \qquad
  & &\hbox{if $\alp = \eps_i - \eps_j$, $\beta = \eps_k - \eps_i$, $i,j,k\ne$}.
\tag{ER4}\end{align}
In particular, the two cases for $\alp \edge (-\beta)$ in \eqref{ano0} correspond
precisely to the relations \eqref{ano3} and \eqref{ano4}.

\subsection{Another look at $\St_N(A)$: fewer generators.} \label{less}
We continue with $N$ and $R$ as in \ref{ano}. In addition we choose a
nontrivial partition
\[ N = I \dotcup J, \qquad \emptyset \ne I \ne N, \]
which we fix in the following.
It induces a non-trivial partition
\begin{equation} \label{less3gr}
    R=R_1 \dotcup R_0 \dotcup R_{-1},
\end{equation}
whose parts are
\begin{align*} R_1 &= \{ \eps_i - \eps_j : i \in I,
\, j \in J \}, \\ R_{-1} &= \{\eps_j - \eps_i : i \in I, j\in J \} = - R_1, \\
R_0 &= \{ \eps_k - \eps_l : (k,l) \in I \times I \hbox{ or } (k,l) \in J
\times J \} = \rma_I \times \rma_J.
\end{align*}
The partition $R=R_1 \dotcup R_0 \dotcup R_{-1}$ will later be seen to be an
example of a $3$-grading of $R$, but we do not need this now. Observe that
every $\mu \in R_0$ can be written (not uniquely) as $\mu = \alp - \beta$
with $\alp$ and $\beta \in R_1$ satisfying $\alp \edge \beta$. Indeed, \sm

\begin{enumerate}[(i)]
  \item \label{less-i} if $\mu = \eps_k - \eps_l$ with $k,l\in I$ then $\mu
      = (\eps_k - \eps_j) - (\eps_l - \eps_j)$ for any $j\in J$, hence
\begin{equation}
  \label{less1} \x_\mu(a) = \x_{kl}(a) = \dd {\x_{kj}(a)} {\x_{jl}(1)}= \dd
        {\x_\al(a)} {\x_{-\beta}(1)}
\end{equation}
by \eqref{ano3} for $\alp= \eps_k - \eps_j$ and $\beta= \eps_l - \eps_j \in
R_1$, and \sm

 \item\label{less-ii}  if $\mu = \eps_k - \eps_l$ with $k,l\in J$ then $\mu
     = (\eps_i - \eps_l) - (\eps_i - \eps_k)$ for any $i\in J$, hence
\begin{equation}
  \label{less2} \x_\mu(a) = \x_{kl}(a) = \dd {\x_{il}(-a)} {\x_{ki}(1)}
       = \dd {\x_\al(-a)} {\x_{-\beta}(1)}
\end{equation}
by \eqref{ano4} for $\alp = \eps_i - \eps_l$ and $\beta = \eps_i - \eps_k
\in R_1$.
\end{enumerate}
\ms

The equations \eqref{less1} and \eqref{less2} show that $\St_N(A)$ is already
generated by \begin{equation}  \{ \x_\al(a) : \alp \in R_1 \cup R_{-1}, \, a
\in A \}. \label{less3} \end{equation}

\subsection{Another look at $\St_N(A)$: fewer relations.} \label{rw}
Our next goal is to rewrite some of the relations \eqref{ano1} --
\eqref{ano4} in terms of the smaller generating set \eqref{less3}. Each of
these relations depend on two roots $\xi$, $\tau \in R$. Because of ${\dd g
h}{}\me = \dd hg$ we only need to consider the relations involving $(\xi,
\tau)$ lying in one of the following subsets of $R \times R$:
\[ R_1 \times R_1, \quad  R_{-1} \times R_{-1}, \quad  R_1 \times R_{-1}, \quad
R_0 \times R_1, \quad  R_0 \times R_{-1}, \quad  R_0 \times R_0.\]%
\sm

\begin{inparaenum}[(a)]
 \item \label{rw-a} Case $(\xi,\tau) = (\alp,\beta) \in R_1 \times R_1$:
Given $\alp, \beta \in R_1$,  exactly one of the relations $\alp = \beta$, $\alp \edge
     \beta$, $\alp \perp \beta$ holds. Hence only \eqref{ano1} and
     \eqref{ano2} apply in this case and yield
     \begin{equation}  \label{rw1}
        \dd {\x_\al(a)} {\x_\beta(b)} = 1 \qquad \hbox{for $\alp$, $\beta \in R_1$ and $a,b\in A$}.
     \end{equation}

\noindent It will now be more convenient to change notation (again) and put
for $\alp = \eps_i - \eps_j\in R_1$ and $u_\al = a E_\al^+$
\begin{align}
  E_\al^+ &= E_{ij} , &
\xpp(u_\al) &= \x_\al(a) = \x_{ij}(a)
 \label{rw-a1} \\
  V_\al^+ &= A E_\al^+, &
  V^+ &= \textstyle \bigoplus_{\alp \in R_1} V_\al^+
     = \bigoplus_{(ij) \in I \times J} \, A E_{ij}. \label{rw-a2}
\end{align}
Because of \eqref{rw1} the map
\begin{equation}\label{xpp-def}  \xpp \co V^+ \longto
    \St_N(A), \quad u = \textstyle \sum_{\alp \in R_1} \, u_\al \; \mapsto \;
    \prod_{\alp \in R_1} \xpp(u_\al)
\end{equation}
is well-defined (independent of the chosen order for $\prod_{\alp \in R_1}$)
and satisfies
\begin{equation} \label{rw2} \xpp(u + u') = \xpp(u) \, \xpp(u')
\qquad \hbox{for $u,u' \in V^+$}.
\end{equation}
It is clear that conversely \eqref{rw2} implies \eqref{rw1}. \ms

\item\label{rw-b}  Case $(\xi,\tau) = (-\alp,-\beta) \in R_{-1} \times
    R_{-1}$: This case is analogous to Case~\eqref{rw-a}. Given $\alp = \eps_i
    - \eps_j \in R_1$ and $v_\alp = a E_\al^-$ we define
\begin{align}
\nonumber  E_\al^- &= E_{ji} , &
\xm (v_\al) &= \x_{-\al}(-a)= \x_{ji}(-a), \\
\label{rw-b2}   V_\al^- &= A E_\al^-, &
  V^- &= \textstyle \bigoplus_{\alp \in R_1} V_\al^-
     = \bigoplus_{(ji) \in J \times I} \, A E_{ji}.
\end{align}
(The minus sign in the definition of $\xm(v_\al)$ is not significant, but has
been included so that the relations below are precisely those used later on. It avoids a minus sign in the formula \eqref{exc2}.)
As in Case~\eqref{rw-a} the relations \eqref{ano1} -- \eqref{ano4} for
$(-\alp,-\beta) \in R_{-1} \times R_{-1}$ yield $\dd {\xmp(v_\al)}
{\xmp(v'\be)} = 1$ and thus give rise to a well-defined map
\begin{equation}
\label{xmp-def}  \xmp \co V^- \longto \St_N(A), \quad
    v = \textstyle \sum_{\alp \in R_1} \, a_\al E_\al^-\; \mapsto \;
       \prod_{\alp \in R_1} \xmp( - a_\al E_\al^-)
       \end{equation} satisfying
\begin{equation} \label{rw3} \xmp(v + v') = \xmp(v) \, \xmp(v')
\qquad \hbox{for $v,v' \in V^-$}.
\end{equation}

At this point we obtain a new generating set of $\St_N(A)$,
\begin{equation} \label{rw0}
 \St_N(A) = \Lan  \xpp(V^+) \cup \xmp(V^-) \Ran,
\end{equation}
\ms

\item\label{rw-c} Case $(\xi,\tau) = (\alp,-\beta) \in R_1 \times R_{-1}$:
    From this case we will only explicitly keep the relation \eqref{ano2},
    which in our new notation says
\begin{equation} \label{rw4}
   \dd {\xpp(u)} {\xmp(v)} = 1 \qquad \hbox{for $(u,v) \in V^+_\al \times V^-_\be$ with
   $\alp \perp \beta$}.
\end{equation}
In the following Case~\eqref{rw-d} we use the relations \eqref{ano3} and \eqref{ano4}
for $(\al,-\beta) \in R_1\times R_{-1}$ in double commutators. \ms

\item \label{rw-d} Case $(\xi, \tau) = (\mu, \ga) \in R_0 \times R_1$: To
    deal with this case, we view the elements of $V^+$ as $I \times
    J$-matrices with only finitely many non-zero entries, as in
    \eqref{rw-a1}. Similarly, elements in $V^-$ are $J\times I$-matrices with
    finitely many non-zero entries. Matrix multiplication of matrices in $V^+
    \times V^- \times V^+$ is then well-defined and yields the {\em Jordan
    triple product\/} $\{ \cdots \}$, i.e., the map
\begin{equation*} \label{rw-d0}
    \{ \cdots \} \co V^+ \times V^- \times V^+ \longto V^+, \quad (x,y,x)
       \mapsto \{ x\, y\, z\} := xyz + zyx.
\end{equation*}
We claim that \eqref{ano2} -- \eqref{ano4} imply
\begin{equation}
  \label{rw-d1}
  \begin{split} & \dd  {\dd{\xpp(u_\al)} {\xmp(v_\be)}} {\xpp(z_\ga)}
   =           \xpp(- \{ u_\al \, v_\be \, z_\ga\}) \\
  & \qquad \hbox{for $\alp, \beta, \ga \in R_1$ with $\alp \edge \beta$ and
  all $u_\al \in V_\al^+$, $v_\be \in V_\be^-$, $z_\ga \in V_\ga^+$}
\end{split} \end{equation}
We prove this by evaluating all possibilities for $\mu = \alp - \beta$ with
$\al, \beta \in R_1$ satisfying $\al \edge \beta$ and $\ga = \eps_r - \eps_s \in R_1$.
By \ref{less}\eqref{less-i} and \ref{less}\eqref{less-ii} there are two cases for such
a representation of $\mu$, discussed below as (I) and (II). \sm

(I) $\alp = \eps_i - \eps_j$, $\beta = \eps_i - \eps_l$ for $i\in I$ and
$j,l\in J$ distinct. Thus $\mu = \alp - \be = \eps_l - \eps_j$. We let $
u_\al = aE_{ij}$, $v_\be = b E_{li}$, $z_\ga = c E_{rs}$. Then, by
\eqref{ano4}  --  \eqref{elg4},
\begin{align*}
 &\dd {\dd{\xpp(u_\al)} {\xmp(v_\be)}} {\xpp(z_\ga)}
   = \dd {\dd{\x_{ij}(a)} {\x_{li}(-b)}} {\x_{rs}(c)}
    = \dd {\x_{lj}(ba)} {\x_{rs}(c)} =: A, \\
&  \{ u_\al \, v_\be \, z_\ga\} = \{aE_{ij} \, bE_{li}\, cE_{rs}\} =
    \de_{sl} \, cba E_{rj} =: B.
\end{align*}
If $l=s$ then, again by \eqref{elg4}, $A=\x_{rj}(-cba)$, while $B= cba\,
E_{rj}$. Otherwise $l \ne s$,  whence $A=1$ by \eqref{elg2} and clearly
$B=0$. This finishes the proof of \eqref{rw-d1} in case (I). \sm

(II) $\alp = \eps_i - \eps_j$, $\beta = \eps_k - \eps_j$ for distinct $i,k\in
I$ and $j\in J$.
This can be shown in the same way as (I).
 \sm

To obtain a slightly simpler version of \eqref{rw-d1} we apply the commutator
formula \[ \dd g {h_1\, h_2} = \dd g {h_1} \cdot \dd g {h_2} \cdot \dd { \dd
{h_2} g } {h_1}\] with $g= \dd {\xpp(u_\al)} {\xmp(v_\be)}$, $h_1 =
\xpp(z_{\ga})$ and $h_2 = \xpp(z_\de)$ for arbitrary $\de \in R_1$. We obtain
$\dd g {h_1\, h_2} = \dd g {h_1} \cdot \dd g {h_2}$, which allows us to
rewrite \eqref{rw-d1} in the form
\begin{equation}
  \label{rw-d2}
  \begin{split} & \dd  {\dd{\xpp(u_\al)} {\xmp(v_\be)}} {\xpp(z)}
   =           \xpp(- \{ u_\al \, v_\be \, z\}) \\
  & \qquad \hbox{for $\alp, \beta\in R_1$ with $\alp \edge \beta$ and
  arbitrary $u_\al \in V_\al^+$, $v_\be \in V_\be^-$, $z \in V^+$.}
\end{split} \end{equation}
\sm

\item\label{rw-e} Case $(\xi,\tau) = (\mu, -\gam) \in R_0 \times R_{-1}$. We
    proceed as in Case~\eqref{rw-d} and define the Jordan triple product
\[
    \{ \cdots \} \co V^- \times V^+ \times V^- \longto V^-, \quad (x,y,x) \mapsto
       \{ x\, y\, z\} := xyz + zyx,
\]
using matrix multiplication in the definition of $\{ \cdots \}$. As in
Case~\eqref{rw-d} one then proves the relation
\begin{equation}
  \label{rw-e2}
  \begin{split} & \dd  {\dd{\xmp(v_\al)} {\xpp(u_\be)}} {\xmp(w)}
   =           \xmp(- \{ v_\al \, u_\be \, w\}) \\
  & \qquad \hbox{for $\alp, \beta\in R_1$ with $\alp \edge \beta$ and
  arbitrary $v_\al \in V_\al^-$, $u_\be \in V_\be^+$, $w \in V^-$.}
\end{split} \end{equation}
\sm

\item Case $(\xi, \tau) \in R_0 \times R_0$: As we will see below, the
    relations involving these $(\xi, \tau)$ are not needed for presenting
    $\St_N(A)$.
\end{inparaenum}
\lv{%%%%%%%%%%%%%%%%%%%%%%%   lv starts %%%%%%%%%%%%%%%%%%%%%%%%%%%%%%%%%%%%%%
We prove \eqref{rw-e} for $w=w_\ga \in V^-_\ga$. We consider Case (I): $\alp = \eps_i
- \eps_j$, $\beta = \eps_i - \eps_l \in R_1$ with $i\in I$ and distinct
$j,l\in J$. We let $\gam = \eps_r - \eps_s$ and $v_\al = a E^-_\al = a
E_{ji}$, $u_\be = b E_\be^+ = bE_{il}$, $w_\ga = cE_\ga^- = c E_{sr}$. Then
\begin{align*}
 &\dd {\dd{\xmp(v_\al)} {\xpp(u_\be)}} {\xmp(w_\ga)}
   = \dd {\dd{\x_{ji}(-a)} {\x_{il}(b)}} {\x_{sr}(-c)}
    = \dd {\x_{jl}(-ab)} {\x_{sr}(-c)} =: A, \\
&  \{ v_\al \, u_\be \, w_\ga\} = \{aE_{ji} \, bE_{il}\, cE_{sr}\} =
    \de_{ls} \, cba E_{jr} =: B.
\end{align*}
If $l=s$ then $A=\x_{jr}(abc) = \xmp( -abc E_{jr}) $, while $B=abc E_{jr}$. If
$l\ne s$ then $A=1$ since $r \ne j$ and $B=0$.}%%%%%%%%%%%%%%%%%%%%%%%%%%%%%%

\subsection{The Steinberg group $\St(\MM_{IJ}(A),\R)$.} \label{news}
We keep the setting of \eqref{ano} -- \eqref{rw}. In \ref{rw}
we defined a pair of matrix spaces,
\[ (V^+, V^-) = \big(\Mat_{IJ}(A), \, \Mat_{JI}(A) \big) =: \MM_{IJ}(A),\]
and Jordan triple products
\[ \{ \cdots \} \co V\si \times V\msi \times V\si \to V\si, \quad (x,y,z) \mapsto
\{x\, y\, z\} = xyz + zyx \] for $\sig\in \{+,- \}$. In \eqref{rw-a2} and \eqref{rw-b2}
we also introduced a family
\[ \R = \textstyle (V_\al)_{\al \in R_1}, \quad V_\al = (V_\al^+, V_\al^-)
\]
of pairs of subgroups with the property that $V\si = \bigoplus_{\al \in R_1} V_\al\si$.
Furthermore, in \eqref{rw0} we found a new generating set for $\St_N(A)$, and
we rewrote some of the relations defining $\St_N(A)$ in terms of this
new generating set. It is then natural to define a new Steinberg group using
these new generators and relations. \sm %

The {\it Steinberg group\/} $\St( \MM_{IJ}(A), \R)$
is the group presented by
\begin{itemize}
  \item[$\bu$] generators $\xp(u)$, $u\in V^+$, and $\xm(v)$, $v\in V^-$,
      and \sm
 \item the relations \eqref{rw2}, \eqref{rw3}, \eqref{rw4}, \eqref{rw-d2}
     and \eqref{rw-e2}. Taking $\sig \in \{+, -\}$ these are
 \begin{align}
   \tag{EJ1}\label{ej1} &\xs (u+ u') = \xs(u) \, \xs(u') \quad  \hbox{for $u, u'\in V\si$}, \\
  \tag{EJ2}\label{ej2}  &\dd {\xp(u)} {\xm(v)} = 1 \quad
              \hbox{for $(u,v) \in V^+_\al \times V^-_\be$, $\alp \perp\beta$}, \\
  \tag{EJ3}\label{ej3}  & \dd  {\dd{\xs(u)} {\xms(v)}} {\xm(z)}  = \xs(- \{ u \, v \, z\}) \\
    & \qquad \hbox{for $u_\al \in V_\al\si$, $v \in V_\be\msi$, $z \in V\si$ with
    $\alp \edge \beta$.} \nonumber
 \end{align}
\end{itemize}
(The letter ``J'' in (EJi) stands for ``Jordan'', to be explained in the next
section.) \sm

From the review above, it is clear that we have a surjective homomorphism of
groups
\begin{equation}
  \label{news1} \Phi \co \St(\MM_{IJ}(A),\R) \to \St_N(A), \quad
       \xs(u) \mapsto \xs'(u)
\end{equation}
where  $\xs'$ is defined in \eqref{xpp-def} and \eqref{xmp-def}. Moreover, composing $\Phi$ with the surjective group homomorphisms $\wp_N \co \St_N(A)
\to \EL_N(A)$ of \eqref{stgwp} yields another surjective group homomorphism
\begin{equation}
 \label{news2} \p_N \co \St(\MM_{IJ}(A), \R) \to \EL_N(A),
  \qquad \xp(aE_{ij}) \mapsto \e_{ij}(a), \quad \xm(aE_{ji}) \mapsto \e_{ji}(-a)
\end{equation}
and hence a commutative diagram
\begin{equation}  \label{news3}  \vcenter{
 \xymatrix{ \St(\MM_{IJ}(A), \R) \ar[rr]^\Phi \ar[dr]_{\p_N} &&
        \St_N(A) \ar[dl]^{\wp_N} \\    & \EL_N(A)
} } \end{equation}
\bigskip

\begin{theorem}[{\cite{LN}}] \label{mth}
The map $\Phi$ of \eqref{news1} is an isomorphism of groups.
\sm

In particular, $\St(\MM_{IJ}(A),\R)$ is centrally closed if $|N|\ge 5$
and  $\p_N$ is a universal central extension of\/ $\EL(A)$ if $N=\NN$.
\end{theorem}

\begin{proof}
To prove bijectivity of $\Phi$ is bijective, a canonical approach is to show
  that the family of $\xp(aE_{ij})$ and $\xm(-bE_{ij} )\in \St(\MM_{IJ}(A),
  \R)$ can be extended to a family of elements satisfying the defining
  relations \eqref{elg1} -- \eqref{elg3} of $\St_N(A)$. As a consequence,
  this yields a group homomorphism $\Psi \co \St_N(A) \to \St(\MM_{IJ}(A),
  \R)$ such that $\Psi \circ \Phi$ and $\Phi \circ \Psi$ are the identity on
  the respective generators and therefore also on the corresponding groups.
  Another proof of the bijectivity of $\Phi$ is given in \cite[24.18]{LN},
  based on the interpretation of both groups as initial objects in an
  appropriate category of groups mapping onto $\EL_N(A)$.

The second part of the theorem follows from the Kervaire-Milnor-Steinberg
Theorem~\ref{kms}. \end{proof}

\subsection{Another look at $\EL_N(A)$.} \label{exc}%\begin{example}
It follows from the existence of the surjective group homomorphism $\p_N$ of
\eqref{news2} that $\EL_N(A)$ is generated by $\p_N(\xp(V^+)\cup \xm(V^-))$
and that the matrices in this image satisfy the relations \eqref{ej1} --
\eqref{ej3}. It is instructive to verify this directly.
\sm

For $(u,v)\in \MM_{IJ}(A)$ we define elements $\e_+(u)$ and $\e_-(v)$ of the
ring $\Mat_N(A)_{\rm ex}$ of \ref{elg} by
\begin{equation} \label{exc0}
   \e_+(u) = \begin{pmatrix} {\bf 1}_I & u \\ 0 & {\bf 1}_J \end{pmatrix},
\qquad
 \e_-(v) = \begin{pmatrix} {\bf 1}_I & 0 \\ -v & {\bf 1}_J \end{pmatrix}.
\end{equation}
Then clearly
\begin{equation}\label{exc1}
  \e_+(u + u') = \e_+(u) \, \e_+(u') \quad \hbox{and} \quad
 \e_-(v + v') = \e_-(v) \, \e_-(v').
\end{equation}
In particular, the matrices $\e_+(u)$ and $\e_-(v)$ are invertible with
inverses $\e_+(u)\me = \e_+(-u)$ and $\e_-(v)\me = \e_-(-v)$. Since
$\e_+(aE_{ij}) = \e_{ij}(a)$ and $\e_-(vE_{ji}) = \e_{ji}(-v)$ for $(ij) \in
I \times J$, the equations \eqref{exc1} also show that $\e_+(u) \in \EL_N(A)$
and $\e_-(v) \in \EL_N(A)$. By straightforward matrix multiplication one
obtains
\begin{equation}\label{exc2}
\dd {\e_+(u)} {\e_-(v)} =
 \begin{pmatrix}{\bf 1}_I - uv+uvuv & uvu \\ vuv& {\bf 1}_J + vu\end{pmatrix}.
\end{equation}
In particular, taking $(u,v)$ with $vu=0$ or $uv=0$, this proves
\[
 \begin{pmatrix}{\bf 1}_I - uv & 0\\ 0& {\bf 1}_J\end{pmatrix} \in \EL_N(A) \quad
 \hbox{and} \quad
 \begin{pmatrix}{\bf 1}_I  & 0\\ 0& {\bf 1}_J +vu\end{pmatrix} \in \EL_N(A).
\]
Specifying $(u,v)$ even more, one then easily sees that all elementary
matrices $\e_{kl}(a)$ with $(k, l)\in I\times I$ or $(k,l)\in J\times J$ lie
in the subgroup of $\EL_N(A)$ generated by $\e_+(V^+) \cup \e_-(V^-)$.
Therefore, this subgroup equals $\EL_N(A)$. \sm

The relation \eqref{ej1} is \eqref{exc1}, and the relation \eqref{ej2}
follows from \eqref{exc2} since for $(u,v) \in V_\al^+ \times V_\al^-$ with
$\alp \perp \beta$ we have $uv=0$ and $vu=0$. In order to prove \eqref{ej3}
in case $\sig=+$, let $(u,v) \in V_\al^+ \times V_\be^-$ with $\alp \perp \beta$ and
let $z\in V^+$ arbitrary. Then $uvu=0=uvzvu$, $vuv=0$ and $({\bf 1}_J +
vu){}\me = {\bf 1}_J - vu$. Hence, by \eqref{exc2},
\begin{align*}
\dd {\dd {\e_+(u)}{\e_-(v)}} {\e_+(z)} &=
\DBL \begin{pmatrix}{\bf 1}_I - uv & 0\\ 0& ({\bf 1}_J -vu)\me\end{pmatrix}, \,
    \begin{pmatrix}{\bf 1}_I & z\\ 0& {\bf 1}_J \end{pmatrix} \DBR
\\
&=\begin{pmatrix}{\bf 1}_I & -z + (1-uv)z(1-vu) \\ 0& {\bf 1}_J \end{pmatrix}
 = \e_+(- \{u\, v\, z\}).
 %\begin{pmatrix}{\bf 1}_I & - \{ u\, v\, z\} \\ 0& {\bf 1}_J \end{pmatrix}
\end{align*}
The relation \eqref{ej3} for $\sig=-$ can be verified in the same way.
\lv{%%%%%%%%%%%%%%%%%   lv
For $(v,u) \in V_\al^- \times V^+_\be$ with $\alp \edge \beta$ and arbitrary
$w\in V^-$ we have $vuv=0= vuwvu$, $uvu=0$ and  $({\bf 1}_I + uv){}\me = {\bf
1}_I - uv$. Hence, by \eqref{exc2},
\begin{align*}
\dd {\dd {\e_-(v)}{\e_+(u)}} {\e_-(w)} &=
\DBL \begin{pmatrix} ({\bf 1}_I - uv)\me & 0\\ 0& {\bf 1}_J -vu \end{pmatrix}, \,
    \begin{pmatrix}{\bf 1}_I & 0\\ -w& {\bf 1}_J \end{pmatrix} \DBR
%\dd {\diag({\bf 1}_I -uv, ({\bf 1}_J - uv)\me )} {\e_+(z)}
\\
&=\begin{pmatrix}{\bf 1}_I & 0 \\ w -(1-vu)w(1-uv)& {\bf 1}_J \end{pmatrix}
 = \begin{pmatrix}{\bf 1}_I & 0\\ vuw+wuv& {\bf 1}_J \end{pmatrix}
 \\ & = \e_-(- \{v\, u\, w\}).
\end{align*}}%%%%%%%%%%%%%  end lv

To put this example in the general setting of the following section \S\ref{sec:gen} we point out that the calculations above are not only valid for matrices of finite or countable size $|N|$, but hold for $N$ of arbitrary cardinality.

%%%%%%%%%%%%%%%%%   end of section 1 %%%%%%%%%%%%%%%%%%%%%%

\section{Generalizations} \label{sec:gen}

In this section we generalize the Steinberg groups considered in
\S\ref{sec:elg}. The generalization has a combinatorial aspect, $3$-graded
root systems, and an algebraic aspect, root graded Jordan pairs. They are
presented in \ref{lfrs} -- \ref{3gra} and \ref{defjp} -- \ref{rgjp}
respectively.  We define the Steinberg group of a root graded Jordan pair
(\ref{stvr}) and state the Jordan pair version of the
Kervaire-Milnor-Steinberg Theorem  (\ref{thmA} and \ref{thmB}). Since the
elementary linear group only makes sense for special Jordan pairs, we replac
it by its central quotient which can be defined for any Jordan pair: the
projective elementary group $\PE(V)$ of a Jordan pair $V$ defined in terms of
the Tits-Kantor-Koecher algebra $\tkk(V)$ (\ref{pev}). We discuss $\tkk(V)$
and $\PE(V)$ for the Jordan pair of rectangular matrices over a ring in
\ref{rect-pvex} and \ref{pvn}. \sm

\subsection{Locally finite root systems \cite{lfrs}}\label{lfrs}
We use $\lan \cdot, \cdot \ran$ to denote the canonical pairing
between a real vector space $X$ of arbitrary dimension and its dual space $X^*$, thus $\lan x, \vphi \ran = \vphi(x)$ for $x\in X$ and $\vphi \in X^*$.  If $\vphi \in X^*$ satisfies $\lan \alp, \vphi\ran = 2$, we define the
{\em reflection\/}  $s_{\alp, \vphi} \in \GL(X)$ by
$$
s_{\alp, \vphi}(x) = x - \lan x, \vphi\ran \alp.
$$

 A {\em locally finite
root system\/} is a pair $(R,X)$ consisting of a real vector space $X$ and a
subset $R\subset X$ satisfying the axioms \eqref{lfrs-i} -- \eqref{lfrs-iv}
below.
\ms
\begin{enumerate}[(i)]
\item\label{lfrs-i}  $R$ spans $X$ as a real vector space and $0 \in R$,

\item \label{lfrs-ii} for every $\alp \in R\ti = R \setminus \{0\}$ there exists $\al\ch \in X^*$  such that $\lan \alp, \alp\ch\ran = 2$ and $s_{\al, \alp\ch}(R) = R$.

\item\label{lfrs-iii}  $\lan \alp, \beta\ch\ran \in \ZZ$ for all $\alp,
    \beta \in R\ti$.

\item\label{lfrs-iv}  $R$ is locally finite in the sense that $R \cap Y$ is
    finite for every finite-dimensional subspace of $X$.
\end{enumerate}
\sm

Locally finite root systems form a category $\RS$, in which a morphism $f\co
(R, X) \to (S, Y)$ is an $\RR$-linear map with $f(R)\subset S$. In this
category, an isomorphism $f\co (R,X) \to (S,Y)$ is a vector space isomorphism
$f\co X \to Y$ with $f(R) = S$. Such an isomorphism necessarily satisfies $\lan f(\alp), f(\beta)\ch\ran = \lan \alp, \beta\ch\ran$ for all $\alp,\beta \in R\ti$. \ms

{\bf Remarks, facts and more definitions.} \begin{inparaenum}[(a)]
  \item The linear form $\alp\ch$ in \eqref{lfrs-ii} is uniquely determined by the two    conditions in \eqref{lfrs-ii}. Therefore, we simply write
$s_\alp$ instead of $s_{\alp, \alp\ch}$ in the future.
\sm

\item Our standard reference for locally finite root systems is \cite{lfrs}.
    As in \cite{lfrs} we will also here abbreviate the term ``locally finite
    root system'' by {\em root system\/}. Then a {\em finite root system\/}
    is a root system $(R,X)$ with $R$ a finite set, equivalently $\dim X
    <\infty$. Finite root systems are the root systems studied for example in
    \cite[Ch.~VI]{brac}. That \cite{brac} assumes $0\notin R$ does not pose any problem in applying the results developed there.

    The real vector space $X$ of a root system $(R,X)$ is usually not
    important. We will therefore often just refer to $R$ rather
    than to $(R,X)$ as  a root system. \sm

 \item As in \cite{lfrs} and again in \cite[\S2]{LN} we assume here that
     $0\in R$, which is more natural from a categorical point of view.
       In \cite[\S2]{LN} the real vector space $X$ is replaced by a
    free abelian group $X$ and condition~\eqref{lfrs-iv} becomes that $R\cap
    Y$ be finite for every finitely generated subgroup $Y$ of $X$. With the
    obvious concept of a morphism, this defines a category of root systems
    over the integers, which is equivalent to the category $\RS$
    \cite[Prop.~2.9]{LN}. \sm

\item A locally finite root system need not be {\em reduced\/} in the sense
    that $\RR \alp \cap R = \{\pm \alp\}$ for every $\alp \in R\ti$. The {\em
    rank\/} of a root system $(R,X)$ is defined as the dimension of the  real vector space $X$. \sm

\item The {\em direct sum\/} of a family $(R^{(j)},
    X^{(j)})_{j\in J}$ of root systems is the pair
    \[ \textstyle \big(\bigcup_{j\in J}  R^{(j)},\;\bigoplus_{j\in J} X^{(j)}\big),
     \]
     which is again a root system
    \cite[3.10]{lfrs}, traditionally written as $R=\bigoplus_{j\in J} R^{(j)}$. A
    non-empty root system is called {\em irreducible\/} if it is not
    isomorphic to a direct sum of two non-empty root systems. Every root
    system uniquely decomposes as a direct sum of irreducible root systems,
    called its {\em irreducible components\/} \cite[3.13]{lfrs}. \sm

\item \label{lfrs-f} Every root system $(R,X)$ admits an inner product
    $\inpr\co X \times X \to \RR$, which is {\em invariant\/} in the sense
    that $(s_\al(x) \mid s_\al(y)) = (x\mid y)$ holds for all $\alp \in R\ti$
    and $x,y\in X$, equivalently
    \begin{equation} \label{lfrs1}
    \lan \beta, \alp\ch \ran = 2\, \frac{(\beta \mid  \alp)}{(\alp \mid  \alp)} \qquad
    \hbox{for all $\alp,\beta \in R\ti$}
    \end{equation}
\cite[4.2]{lfrs}. If $R$ is irreducible, $\inpr$ is unique up to a non-zero
scalar. It follows that the definition of a root system given in \cite{n:cr}
is equivalent to the definition above, and that a finite reduced root system
is the same as a ``root system'' in \cite{hum}, again up to $0 \notin R$.
\end{inparaenum}

\subsection{Classification of root systems} \label{class} We first present,
as examples, the {\em classical root systems\/} $\rma_I, \ldots,
\rmbc_I$. Let $I$ be a set of cardinality $|I|\ge 2$ and let $X=
\bigoplus_{i\in I} \RR\eps_i$ be the $\RR$-vector space with basis
$(\eps_i)_{i\in I}$. Define
\begin{align}
  \label{class-A} \rma_I &= \{ \eps_i - \eps_j : i,j\in I\}, \\
  \label{class-D} \rmd_I &= \rma_I \cup \{ \pm (\eps_i + \eps_j) : i \ne j \},\\
   \label{class-B} \rmb_I &= \rmd_I \cup \{ \pm \eps_i : i \in I\}, \\
\label{class-C} \rmc_I &=\rmd_I \cup \{ \pm 2\eps_i : i \in I\}, \\
\label{class-BC} \rmbc_I &= \rmb_I \cup \rmc_I.
\end{align}

Then $\rma_I$ is a root system in $\dot X = \Ker (t)$
where $t\in X^*$ is defined
by $t(\eps_i) = 1$, $i\in I$. Its rank is therefore $|I|-1$. The notation
$\rma$ instead of the traditional $\rmaf$ is meant to indicate this fact.
Observe that $\rma_\NN$ is the root system $R$ of \ref{ano}. All other sets
$\rmd_I, \ldots, \rmbc_I$,  are root systems in $X$, whence of rank $|I|$.
The root systems $\rma_I$, $\rmb_I$, $\rmc_I$ and $\rmd_I$ are reduced, while
$\rmbc_I$ is not.

The isomorphism class of a classical root system only depends on the
cardinality of the set $I$. In particular, when $I$ is finite of cardinality $n$
we will use the index $n$ instead of $I$. Thus, $\rmd_n = \rmd_{\{1, \ldots,
n\}}$ etc. Note $\rma_{n+1} = \rma_{\{0,1,\ldots, n\}} = \rmaf_n$ in the
traditional notation.

The standard inner product $\inpr$, defined by $(\eps_i | \eps_j) =
\del_{ij}$, is an invariant inner product in the sense of
\ref{lfrs}\eqref{lfrs-f}. With the exception of $\rmd_2 = \rmaf_1 \oplus
\rmaf_1$, the root systems $\rma_I, \ldots, \rmbc_I$ are irreducible.
Apart from the low-rank isomorphisms $\rmb_2 \cong \rmc_2$, $\rmd_3 \cong
\rmaf_3$, they are pairwise non-isomorphic. \sm

The classification of root systems \cite[Thm.~8.4]{lfrs} says that {\em an
irreducible root system is either isomorphic to a classical root system or to
an exceptional finite root system.}

\subsection{$3$-graded root systems.} \label{3gra} A {\em $3$-grading\/} of a
root system $(R,X)$ is a partition $R = R_1 \dotcup R_0 \dotcup R_{-1}$
satisfying  the following conditions \eqref{3gra-i} -- \eqref{3gra-iii} below:

\begin{enumerate}[(i)]
\item \label{3gra-i} $R_{-1} = - R_1$;

\item\label{3gra-ii}  $(R_i + R_j) \cap R \subset R_{i+j}$ for $i,j\in \{1,
    0, -1\}$, with the understanding that $R_k = \emptyset$ for $k\notin
    \{1, 0, -1\}$

\item\label{3gra-iii} $R_1 + R_{-1} = R_0$, i.e., every root in $R_0$ is a
    difference of two roots in $R_1$.
\end{enumerate}
In particular \eqref{3gra-ii} says that the sum of two roots in $R_1$ is
    never a root and $(R_1 + R_{-1}) \cap R \subset R_0$, a condition
    which is strengthened in \eqref{3gra-iii}. \sm

Since a $3$-grading of a root system $(R,X)$ is determined by the subset
$R_1$ of $R$, we will denote a $3$-graded root system by $(R,R_1, X)$ or simply by $(R, R_1)$. A {\em $3$-graded root system\/}  is a root system equipped with a
$3$-grading. An {\em isomorphism\/} $f \co (R, R_1, X)  \to (S, S_1, Y)$
between $3$-graded root systems is a vector space isomorphism $f\co X \to Y$
satisfying $f(R_1) = S_1$, hence also $f(R_i) = S_i$ for $i\in \{1,0,-1\}$,
and is therefore an isomorphism $f\co(R, X) \to (S, Y)$ of the underlying
root systems. References for $3$-graded root systems are \cite[\S17,
\S18]{lfrs}, \cite[Ch.~IV]{LN} and \cite{n:cr}. \ms

{\bf Some facts and examples.} \begin{inparaenum}[(a)]
  \item \label{3gra-a} The decomposition \eqref{less3gr} of the root system
      $R=\rma_\NN$ is a $3$-grading. The restrictions on $N$ imposed in
      \ref{ano} are not necessary for the definition of a $3$-graded root
      system, as we have seen in \ref{class} for the root system $\rma_N$.
      Any non-trivial partition $N=I \dotcup J$ induces a $3$-grading of
      $\rma_N$ as in \ref{less}, denoted $\rma_N^I$. Thus, the $1$-part of
      the $3$-graded root system $\rma_N^I$ is
      \begin{equation}\label{3gra-a1}
       ( \rma_N^I)_1 = \{ \eps_i - \eps_j : i\in I , j\in N \setminus I\}.
\end{equation}
      Every $3$-grading of $\rma_N$ is obtained in this way for a non-empty
      proper subset $I\subset N$. \sm

  \item \label{3gra-bb} A $3$-grading $\rmb_I\qf$ of the root system $\rmb_I$
      is obtained by choosing a distinguished element of $I$, say $0\in I$,
      and putting $R_1 = \{ \eps_0 \} \cup \{ \eps_0 \pm \eps_i : 0 \ne i\in
      I \}$. \sm

  \item \label{3gra-b} The root system $R=\rmc_I$ has a $3$-grading, denoted
      $\rmc_I\her$, whose $1$-part is $ R_1 = \{ \eps_i + \eps_j : i,j\in
      I\}$.  Note $R_0 = \{\eps_i - \eps_j : i, j \in I\} \cong \rma_I$. \sm

 \item The root system $\rmd_I$, $|I|\ge 4$, is a subsystem of $\rmb_I$ and
     $\rmc_I$. The $3$-gradings of these two root systems, defined in
     \eqref{3gra-bb} and \eqref{3gra-b}, induce $3$-gradings $\rmd_I\qf$ and
     $\rmd_I\alt$ of $\rmd_I$. The first of these is determined by $R_1 = \{
     \eps_0 \pm \eps_i : 0 \ne i\in I \}$ and the second by $R_1 = \{ \eps_i
     + \eps_j : i,j\in I , i \ne j \}$. It is known that $\rmd_I\qf \cong
     \rmd_I\alt$ if $|I|=4$, but $\rmd_I\qf \not\cong \rmd_I\alt$ if $|I| \ge
     5$.
   \sm

\item \label{3gra-c} The $3$-gradings of a root system $R$ are determined by
    the $3$-gradings of its irreducible components $(R^{(j)})_{j\in J}$ as
    follows.

    If $(R, R_1)$ is a $3$-grading, then $(R^{(j)}, R_1 \cap
    R^{(j)})$ is a $3$-grading for every $j\in J$. Conversely, given
    $3$-gradings $(R^{(j)}, R_1^{(j)})$ for every $j$, the set $R_1 =
    \bigcup_j R_1^{(j)}$ defines a $3$-grading of $R$.

   These easy observations reduce the classification of $3$-graded root
   systems to the case of irreducible root systems. Their
   classification is given in \cite[17.8, 17.9]{lfrs}. It turns out that an
   irreducible root system has a $3$-grading if and only if it is not
   isomorphic to $\rme_8$, $\rmf_4$ and $\rmg_2$. Some irreducible root
   systems have several non-isomorphic $3$-gradings, such as $\rma_N$ or
   $\rmd_I$. But every $3$-grading of $\rmc_I$ is isomorphic to the
   $3$-grading $\rmc_I\her$ of \eqref{3gra-b}. \sm

\item The relations $\perp$ and $\edge$ introduced in \eqref{anopdef} in case
    $R=\rma_\NN$ can be defined for any root system $R$ without using an
    invariant inner product. For $\alp, \beta \in R\ti$ we put
\begin{align*}
  \alp \perp \beta \quad &\iff \quad \lan \alp, \beta\ch\ran = 0, \quad \hbox{equivalently
    $\lan \beta, \alp\ch \ran = 0$,}\\
   \alp \edge \beta \quad &\iff \quad \lan \alp, \beta\ch\ran = 1 = \lan \beta, \alp \ch\ran \\
  \alp \lra \beta \quad &\iff \quad \lan \alp, \beta\ch\ran = 2, \; \lan\beta, \alp\ch\ran = 1.
\end{align*}
The formula \eqref{lfrs1} shows that the definitions of $\perp$ and $\edge$
above generalize \eqref{anopdef}. The relation $\lra$ occurs for example in
the root system $\rmc_I$: we have $2\eps_i \lra \eps_i + \eps_j$ for $\ne j$.
(In \cite{lfrs} the notation $\coll$ and $\vog$ is used in place of
$\edge$ and $\lra$, respectively.) 
\sm

\item
Given $\alp$, $\beta
    \in R_1$, exactly one of these relations holds:
\begin{equation}\label{3gra1}
\alp = \beta \quad \hbox{or}\quad \alp \lra \beta \quad\hbox{or} \quad
\alp \lla \beta \quad \hbox{or} \quad \alp \edge \beta \quad \hbox{or} \quad \alp \perp \beta.
\end{equation}
Moreover, again for $\alp, \beta \in R_1$,
\begin{equation} \label{3gra2}
  2\alp- \beta \in R \quad \iff \quad 2\alp - \beta \in R_1 \quad \iff \quad \alp \lla \beta
\hbox{ or $\alp = \beta$.}\end{equation}
\end{inparaenum}
%%%%%%%%%%%%%%%%%%%

\subsection{Jordan pairs.}\label{defjp} This subsection contains a very short
introduction to Jordan pairs over a commutative ring $k$,
although for the purpose of this paper $k=\ZZ$ is completely sufficient.
%\comment
%den folgenden Satz weglassen?
%``We will only present what is needed to understand this paper.''
%\endcomment
We will only present what is needed to understand this paper. A more detailed
but still concise introduction to Jordan pairs is given in \cite[\S6]{LN};
the standard reference for Jordan pairs is \cite{jp}. \sm

We have already seen an example of a Jordan pair in \ref{rw}: the rectangular
matrix pair $\MM_{IJ}(A) = (\Mat_{IJ}(A), \Mat_{JI}(A))$ equipped with the
Jordan triple product $\{x\, y\, z\} = xyz+zyx$ for $(x,y,z) \in V\si \times
V\msi \times V\si$ and $\sig=\pm$. The Jordan triple product is the
linearization with respect to $x$ of the expression $Q(x)y = xyx$, which did not play any role in \S\ref{sec:elg}, but which is the basic structure underlying Jordan pairs.
\sm

A {\em Jordan pair\/} is a pair $V=(V^+, V^-)$ of $k$-modules together with maps
\[
    Q\si \co V\si \times V\msi \to V\si, \quad (x,y) \mapsto Q\si(x)y, \qquad (\sig = \pm),
\]
which are quadratic in $x$ and linear in $y$ and which satisfy the identities
\eqref{jp1} -- \eqref{jp3} below in all base ring extensions. To define these
identities, we will simplify the notation and omit $\sig$, thus writing
$Q(x)y$ or simply $Q_x y$. This does not lead to any confusion, as long as one takes care that the
expressions make sense. Linearizing $Q_xy$ in $x$ gives
\[ Q_{x,z}y = Q(x,z)y = Q_{x+z}y - Q_x y - Q_z y,
\]
which we use to define the {\em Jordan triple product\/}
\[
   \{\cdots \} \co V\si \times V\msi \times V\si \to V\si, \quad (x,y,z) \mapsto
     \{x\, y\, z\} = Q_{x,z}y.
\]
To improve readability we will sometimes write $\{x, \, y, \, z\}$ instead of
$\{x\, y\, z\}$. If $K$ is a commutative associative unital
$k$-algebra we let $V\si_K = V\si\ot_k K$ and observe that there exist unique extensions of the $Q\si$ to quadratic-linear maps $Q \co V_K \times V_K \to V_K$.
The identities
required to hold for $x,z\in V_K\si$, $y,v\in V_K\msi$,
$\sig \in \{+,-\}$ and any $K$ as above are
\begin{align}
 \tag{JP1} \label{jp1} \{x, \, y, \, Q_x v\} &= Q_x\, \{y,\, x, \, v\},\\
 \tag{JP2}\label{jp2} \{Q_x y,\, y, \, z\} &= \{x,\, Q_yx, \, z\}, \\
 \tag{JP3}\label{jp3} Q_{Q_xy} v &= Q_x \, Q_y \, Q_z v.
\end{align}
A {\em homomorphism\/} $f\co V\to W$ of Jordan pairs is a pair  $f=(f_+,
f_-)$ of $k$-linear maps $f\usi \co V\si \to W\si$ satisfying $f\usi (
Q(x)y ) = Q\big( f\usi(x)\big) f\umsi(y)$ for all $(x,y) \in V\si \times V\msi$
and $\sig = \pm$. \ms

{\bf Remarks and more definitions.} \begin{inparaenum}[(a)] \item Instead of
requiring that \eqref{jp1} -- \eqref{jp3} hold for all extensions $K$, one
can demand that \eqref{jp1} -- \eqref{jp3} as well as all their
linearizations hold in $V$. For example, linearizing the identity \eqref{jp1}
with respect to $x$ gives the identity
\[  \{z, \, y, \, Q_x v\} +  \{x, \, y, \,  Q_{x,z} v\} = Q_{x,z} \, \{y,\, x, \, v\}
   + Q_x\, \{y, \, z, \, v\}
\]

\item \label{defjp-sub} If $V=(V^+, V^-)$ is a Jordan pair and $S=(S^+, S^-)$
    is a pair of submodules of $V$ satisfying $Q(S\si)S\msi \subset S\si$ for
    $\sig = \pm$, then $S$ is a Jordan pair with the induced operations,
    called a {\em subpair\/} of $V$. \sm

\item \label{defjp-idem} An {\em idempotent\/} in a Jordan pair $V$ is a
    pair $e=(e_+, e_-)\in V$
satisfying $Q(e_+)e_- = e_+$ and $Q(e_-)e_+ =
    e_-$. An idempotent $e$ gives rise to the {\em Peirce decomposition\/} of
    $V$,
\[ V\si = V\si_2(e) \oplus V\si_1(e) \oplus V\si_0(e), \qquad \sig = \pm, \]
where the {\em Peirce spaces\/} $V_i(e)=(V_i^+(e), V_i^-(e))$, $i=0,1,2$,
are given by
\begin{align*}
  V\si_2(e) &= \{ x\in V\si : Q(e\usi) Q(e\umsi)x = x \}, \\
   V\si_1(e) &= \{ x\in V\si : \{e\usi\, e\umsi\, x\} = x \}, \\
   V\si_0(e) &= \{ x\in V\si : Q(e\usi) x = 0 = \{e\usi\, e\umsi\, x\}\} .
\end{align*}
The Peirce spaces $V_i^\pm = V^\pm_i(e)$ satisfy the multiplication rules \begin{align}
Q(V_i\si)V_j\msi &\subset V\si_{2i-j}, & \{V\si_i\, V\msi_j\, V\si_l\}
&\subset V\si_{i-j+l}, \\
\{V_2\si \, V_0\msi\, V\si \} &=  0 = \{V_0\si\, V_2\msi\, V\si\},
\end{align}
where $i,j,l\in \{0,1,2\}$, with the understanding that $V\si_m = 0$ if $m \notin
\{0,1,2\}$. In particular, the $V_i=V_i(e)$ are subpairs of $V$. If $2 \in k\ti$ we have $V\si_i(e) = \{x \in V : \{e\si\, e\msi\, x\} = i x \}$ for $i=0,1,2$.
\end{inparaenum}

\subsection{Examples of Jordan pairs.}\label{exjp}
We now give concrete examples of Jordan pairs, to illustrate the abstract definition of \ref{defjp}. \sm

\begin{inparaenum}[(i)] \item \label{defjp-i} Any associative, not necessarily unital or
commutative $k$-algebra $A$ gives rise to a Jordan pair $V=(A,A)$ with
respect to the operations $Q_xy = xyx$.

Indeed, since a base ring extension of $A$ is again associative,
it suffices to verify the identities in $V$, where they
follows from the following calculations.
\begin{align*}
 \{x,y,Q_xv\} &= xy (xvx) + (xvx)yx = x(yxv + vxy)x = Q_x \{y, v, x\}, \\
 \{Q_xy, y,z\} & = (xyx)yz + zy(xyx) = x(yxy)z + z (yxy) x = \{x, Q_yx, z\}, \\
     Q_{Q_xy} v & = (xyx)v(xyx) = x (y (xvx) y) x = Q_x Q_y Q_x v.
\end{align*}

An idempotent $c$ of the associative algebra $A$, defined by $c^2=c$, induces
an associative Peirce decomposition $A=A_{11} \oplus A_{10} \oplus A_{01}
\oplus A_{00}$ with $A_{ij}= \{a\in A: ca= ia, \, ac=ja\}$. The pair
$e=(c,c)$ is an idempotent of the Jordan pair $V$ whose Peirce spaces are
$V_2\si(e) = A_{11}$, $V_1\si(e) = A_{10} \oplus A_{01}$ and $ V\si_0(e) =
A_{00}$. Not only $V_1(e)$ but also $(A_{10}, A_{01})$ and $(A_{01}, A_{10})$
are subpairs of $V$.

Not every  idempotent of $V$ has the form $(c,c)$, $c$ an idempotent of $A$. For example, if $u\in A\ti$ and $c$ is an idempotent of $A$, then $(uc, c u^{-1})$ is an idempotent of $V$ which, however, has the same Peirce spaces as $(c,c)$.
\sm

\item \label{defjp-ii} By \eqref{defjp-sub} and \eqref{defjp-i}, any pair
    $(S^+,S^-) \subset (A, A)$ of $k$-submodules closed under the operation
    $(x,y) \mapsto xyx$ is also a Jordan pair. Jordan pairs of this form are
    called {\em special}. Their Jordan triple product is
\begin{equation} \label{defjo-ii1}
  \{x\, y \, z \} = xyz + zyx \qquad (x,z\in S\si, y \in S\msi).
 \end{equation}
We next describe some important cases of special Jordan pairs. \sm

   \item \label{defjp-iir} Let $I$ and $J$ be non-empty sets. Let $N =
       I \dotcup J'$ where $J'$ is a set disjoint from $I$ and in bijection
       with $J$ under $j \mapsto j'$, and embed $\Mat_{IJ}(A)$ into the right
       upper corner of the associative algebra $\Mat_N(A)$. Similarly, we
       identify $\Mat_{JI}(A)$ with the left lower corner of $\Mat_N(A)$.
       Then
       \[ \MM_{IJ}(A) = \big( \Mat_{IJ}(A), \Mat_{JI}(A)\big) \]
is a subpair of the Jordan pair $\big( \Mat_N(A), \Mat_N(A)\big)$  and is
therefore a Jordan pair, as claimed at the  beginning of this subsection.
If $N$ is finite, $\MM_{IJ}(A)$ is of type $(A_{10}, A_{01})$ for $A=\Mat_N(A)$ and the idempotent $c= {\bf 1}_I$, see \eqref{defjp-i}.
\sm

   \item\label{defjp-iii} Let $a \mapsto a^J$ be an involution of the
       associative $k$-algebra $A$. Then $\rmH(A,J) = \{a \in A: a^J = a\}$
       is closed under the Jordan pair product, whence $\HH(A,J) = (\rmH(A,
       J), \, \rmH(A, J))$ is a Jordan pair.

   More generally, extend $J$ to an
       involution of the associative $k$-algebra $\Mat_I(A)$, $|I|\ge 2$,
defined by $(x_{ij})^J=(x_{ji}^J)$ and again denoted by
       $J$.   Then the {\em hermitian matrix pair}
   \[   \HH_I(A,J) = \big( \rmH(\Mat_I(A), J), \, \rmH(\Mat_I(A), J)\big)
\]
is a special Jordan pair. In particular, taking $A=k$, $J=\Id_k$ and $I=\{1, \ldots , n \}$ we get  the {\em symmetric matrix pair\/} $\HH_n(k) = (\rmH_n(k), \rmH_n(k))$.
\sm

\item \label{defjp-iv} Let $\Alt_I(k)$ be the alternating $I\times I$-matrices over $k$, where $x=(x_{ij})$ is called alternating if $x_{ii} = 0 = x_{ij} + x_{ji}$ for $i,j\in I$. Then the {\em alternating matrix pair\/} $\bbA_I(k) =     (\Alt_I(k), \Alt_I(k))$ is a subpair of
$\MM_I(k)$, whence a special Jordan pair. \sm

   \item \label{defjp-v} Let $M$ be a $k$-module and let $q\co M \to k$ be a quadratic form with polar   form $b$, defined by $b(x,y) = q(x+y) - q(x) - q(y)$. Then   $\JJ(M,q)=(M,M)$ is a Jordan pair with quadratic operators $Q_xy =   b(x,y)x - q(x)y$. \sm

\item Let $J$ be a unital quadratic Jordan algebra with quadratic operators $U_x$,     $x\in J$ (\cite{jake:tata,jake:ark}).  Then $(J,J)$ is a Jordan pair with quadratic maps $Q_x = U_x$. For example, if $k$ is a field, the rectangular matrix pair $\MM_{pp}(k)$ is  of this form, but $\MM_{pq}(k)$ for $p\ne q$ is not. Thus, there are  ``more'' Jordan pairs than Jordan algebras.

    Let $\scC$ be an octonion $k$-algebra, see for example \cite{SV} in case $k$ is a  field or \cite{lopera} in general,
    and let $J=\rmH_3(\scC)$ be the exceptional Jordan algebra
    of $3\times 3$ matrices over $\scC$ which are hermitian with respect to
    the standard involution of $\scC$. Then $(J,J)$ is a Jordan pair,  which
    is not special in the sense of \eqref{defjp-ii}.
 Such Jordan pairs are called {\em exceptional\/}.
\end{inparaenum}

\subsection{Root graded Jordan pairs.} \label{rgjp} Let us first recast the
Peirce decomposition \ref{defjp}\eqref{defjp-idem} of an idempotent $e$ in a
Jordan pair $V$ from the point of view of a grading.
%\comment
%``a grading''.  Man fragt sich sofort:  which grading?
%\endcomment

We use the $3$-graded root system $\rmc_I\her$ of \ref{3gra}\eqref{3gra-b} with $I=\{0,1\}$. Its $1$-part is $R_1 = \{ \eps_i + \eps_j : i,j\in \{0,1\}\} = \{2 \eps_1, \eps_1 + \eps_0, 2\eps_0\}$. Putting
\[ V\si_\al = V\si_{i+j}(e) \qquad (\alp = \eps_i + \eps_j \in R_1)
\]
we have the decomposition $V\si= \bigoplus_{\alp \in R_1} \,
V\si_\al$ which satisfies
\begin{align}
\tag{RG1} \label{rg1}  Q(V_\al\si)V_\beta\msi &\subset V\si_{2\alp-\beta}, \qquad \qquad  \{V\si_\alp\, V\msi_\beta\,
    V\si_\ga\} \subset V\si_{\alp-\beta + \gam}, \\
\tag{RG2} \label{rg2}
\{V_\al\si \, V_\beta\msi\, V\si \} &=  0 \qquad \qquad \qquad \quad \hbox{if $\alp\perp\beta$}.
\end{align}
Here $2\alp-\beta$ and $\alp-\beta+ \gam$ in \eqref{rg1} are calculated
in  $X= \RR\cdot \eps_0 \oplus \RR \cdot \eps_1 \cong \RR^2$, and $V\si_{2\alp-\beta} = 0$ if $2\alp-\beta \notin R_1$ or $V\si_{\alp -\beta + \gam}=0$ if $\alp-\beta + \gam \not\in R_1$. We see that, apart from the actual definition of the Peirce spaces, the rules
governing the Peirce decomposition can be completely described in terms of
$R_1$. The following generalisation is then natural. \sm

Given a $3$-graded root system $(R,R_1)$ and a Jordan pair $V$, an {\em
$(R,R_1)$-grading of $V$\/} is a decomposition $V\si= \textstyle
\bigoplus_{\alp \in R_1} \, V\si_\al$, $\sig=\pm$, satisfying \eqref{rg1} and
\eqref{rg2}. We will use $\R=(V_\al)_{\alp \in R_1}$ to denote such a
grading. A {\em root graded Jordan pair\/} is a Jordan pair equipped with an
$(R,R_1)$-grading for some $3$-graded root system. In view of \eqref{3gra2}
we can make the first inclusion in \eqref{rg1} more precise:
\begin{equation} \label{rgjp4}
\begin{split}
&Q(V\si_\al) V\msi_\beta = 0 \quad \hbox{\em unless $\alp \lra \beta$, in which case} \\
&\quad \hbox{\em $2\alp - \beta \in R_1$ and $Q(V\si_\al) V\msi_\beta\subset
  V\si_{2\alp - \beta}$.}
\end{split} \end{equation}
\sm

We call $\R$ an {\em idempotent root grading\/} if there exists
 a subset $\De \subset R_1$ and a family $(e_\al)_{\alp \in \Del}$ of non-zero idempotents $e_\al \in V_\al$ such that the  $V_\be$ are given by
\begin{equation}\label{rgjp3}
V_\beta  = \bigcap_{\alp \in \Del}\, V_{\lan \beta, \alp\ch\ran}(e_\al)
\end{equation}
Observe that \eqref{rgjp3} makes sense since $\lan \alp,\beta\ch\ran \in
\{0,1,2\}$ by \eqref{3gra1}.  Neither the idempotents nor the subset $\Del\subset R_1$ are uniquely determined by an idempotent root grading, see for
example \eqref{rgjp-iii} below.

To avoid some technicalities, we will often
assume that $\R$ is a {\em fully idempotent root grading\/}, i.e.,
$\R$ is idempotent with respect to a family of idempotents with $\Del = R_1$. In the terminology of \cite{n:grids} this means that $V$ is covered by
the cog $(e_\al)_{\alp \in R_1}$. \ms

{\bf Examples.} \begin{inparaenum}[(i)] \item Let $R=\rmaf_1= \{\alp, -
\alp\}$ equipped with the $3$-grading defined by $R_1 = \{\alp\}$. An
$(R,R_1)$-graded Jordan pair is simply a Jordan pair $V=(V^+, V^-)$ for which
$V\si = V_\al\si$. This root grading is idempotent if and only if $V\cong (J,J)$
where $J$ is a unital Jordan algebra. To see sufficiency in case $J$ is a
Jordan algebra with identity element $1_J$, one uses $e_\al = (1_J, 1_J)$
and observes $(J, J) = V_2(e_\al)$. \sm

\item Let $(R,R_1) = \rmc_2\her$. We have seen above that the Peirce
    decomposition of an idempotent $e\in V$ can be viewed as a
    $\rmc_2\her$-grading, which is idempotent with respect to $e=e_\al$,
    $\alp = 2\eps_1$. Thus here $\Del = \{\alp\}$. \sm

\item \label{rgjp-iii} Let $(R, R_1) = \rma_N^I$ be the $3$-graded root
    system of \ref{3gra}\eqref{3gra-a}. Put $J = N\setminus I$. An
    $\rma_I^N$-grading of a Jordan pair $V$ is a decomposition
    \[  V=\textstyle \bigoplus_{(i j) \in I \times J} \, V_{(ij)}\]
     such that for all $(ij)$ and $(lm) \in I\times J$ and $\sigma =\pm$ we
     have, defining $V_{(ij)} = V_{\eps_i-\eps_j}$ for $\eps_i-\eps_j \in
     R_1$,
\begin{align*}
Q(V\si_{(ij)})V\msi_{(ij)}& \subset V\si_{(ij)}, &
\{V\si_{(ij)} \, V\msi_{(ij)} \, V\si_{(im)}\}&  \subset V\si_{(im)},
\\
\{ V\si_{(ij)} \, V\msi_{(ij)} \, V\si_{(lj)}\} & \subset
V\si_{(lj)}, & \{V\si_{(ij)} \, V\msi_{(lj)} \, V\si_{(lm)}\}
& \subset V\si_{(im)},
\end{align*}
and all other types of products vanish. \sm

An example of an $\rma_I^N$-graded Jordan pair $V$ is the rectangular matrix
pair $\MM_{IJ}(A)$ of an associative unital $k$-algebra $A\ne 0$, see
\ref{exjp}\eqref{defjp-iir}, with respect to the subpairs $V_{(ij)} =
(AE_{ij}, AE_{ji})$. This $\rma_N^I$-grading of $\MM_{IJ}(A)$ is fully
idempotent with respect to the family $(e_\al)_{\alp \in R_1}$, $e_\al = (
a_{ij}E_{ij}, a\me_{ij}E_{ji})$, where $\alp = \eps_i - \eps_j$ and $a_{ij} \in
A\ti$.
It is also idempotent with respect to the following
smaller family:  fix $i_0\in I$ and $j_0 \in J$ and consider $(e_\al)_{\alp \in \Del}$ where $\Del = \{\eps_i - \eps_{j_0} :
i\in I\} \cup \{ \eps_{i_0} - \eps_j : j\in J\}$.
\sm

Let $\fra, \frb \subset A$ be $k$-submodules with $\fra \frb \fra \subset
\fra$ and $\frb \fra \frb \subset \frb$.  Then $S = (\Mat_{IJ}(\fra),
\Mat_{JI}(\frb))$ is a Jordan subpair of $\MM_{IJ}(A)$. It inherits the
$\rma_N^I$-grading from $\MM_{IJ}(A)$ by putting $S_{(ij)} = (\fra E_{ij},
\frb E_{ji})$. This $\rma_N^I$-grading of $S$ is in general not idempotent,
e.g., if $\fra$ is a nil ideal. \ms

\item The hermitian matrix pair $V=\HH_I(A,J)$ of \ref{exjp}\eqref{defjp-iii}
    has an idempotent $\rmc_I\her$-grading.
Indeed, define
\[  h_{ij}(a) = \begin{cases} a E_{ij} + a^JE_{ji} &\hbox{if $a \in A$ and $i \ne j$}, \\
                               aE_{ii} &\hbox{if $i=j$ and $a\in \rmH(A,J)$.} \end{cases}
\]
Then $V=\bigoplus_{\alp \in R_1} V_\al$ with
\[ V_\al = V_{\eps_i + \eps_j} =
   \begin{cases}\big( h_{ij}(A), h_{ij}(A)\big), & \hbox{if $i \ne j$} \\
                      \big( h_{ii}(\rmH(A,J)), \, h_{ii}(\rmH(A,J))\big),
                          &\hbox{if $i=j$}. \end{cases}
\]
is a $\rmc_I\her$-grading of $V$. It is fully idempotent, for example with respect
to the family $(e_\al)_{\alp \in R_1}$ for which $e_\al = (h_{ij}(1),
h_{ij}(1))$, $\alp = \eps_i + \eps_j$. \ms

\item The remaining examples in \ref{exjp} all have idempotent root gradings.
The alternating matrix pair $\bbA_I(k)$  of \ref{exjp}\eqref{defjp-iv}
    has an idempotent $\rmd_I\alt$-grading (\cite[23.24]{LN}).
 The Jordan pair $\JJ(M,q)$ associated with a quadratic form $q$ in
\ref{exjp}\eqref{defjp-iv} has an idempotent $\rmb_I\qf$ if $q$ contains a hyperbolic plane, or even a $\rmd_I\qf$-grading if $q$ is hyperbolic (\cite[23.25]{LN}). If $\scC$ is a split octonion algebra  in the sense of \cite{SV} or \cite{lopera} the exceptional Jordan pair
    $(\rmH_3(\scC), \rmH_3(\scC))$ has an idempotent root grading with $R$ of
    type $\rme_7$ (\cite[III, \S3]{n:grids}).
\end{inparaenum}

\subsection{The Steinberg group $\St(V,\R)$.}\label{stvr} Let $(R,R_1)$ be a
$3$-graded root system and let $V$ be a Jordan pair with a root grading
$\R=(V_\al)_{\alp \in R_1}$, not necessarily idempotent. The {\em Steinberg
group $\St(V,\R)$\/} is the group with the following presentation: \sm
\begin{itemize}
  \item[$\bu$] The generators are $\xp(u)$, $u\in V^+$, and $\xm(v)$, $v\in V^-$.
\sm

\noindent To formulate the relations, we first introduce, for $\alp \ne \beta \in
R_1$ and $(u,v) \in V_\al^+\times V_\be^-$, the element $\b(u,v)$  in the free
group with the above generators
 by the equation
\begin{equation}\label{stvr1}
    \xp(u) \; \xm(v) = \xm(v + Q_v u)\; \b(u,v) \; \xp(u + Q_u v)
\end{equation}

\item Then the relations are
 \begin{align}
   \tag{St1}\label{sj0} &\xs (u+ u') = \xs(u) \, \xs(u') \qquad \qquad \; \hbox{for $u, u'\in V\si$}, \\
  \tag{St2}\label{sj1}  &\dd {\xp(u)} {\xm(v)} = 1 \qquad \qquad \qquad  \quad
              \hbox{for $(u,v) \in V^+_\al \times V^-_\be$, $\alp \perp\beta$}, \\
  \tag{St3}\label{sj2}
   & \begin{cases}
     \dd {\b(u,v)} {\xp(z)}  = \xp( - \{u\, v\, z\} + Q_u Q_v z), \\
         \dd {\b(u,v)\me} {\xm(y)} = \xm( - \{v\, u\, y\} + Q_v Q_u y)
      \end{cases} \\ \nonumber
      & \qquad \hbox{for all $(u,v) \in V_\al^+ \times V_\be^-$ with $\alp
       \ne \beta$ and all $(z,y) \in V$.}
 \end{align}\end{itemize}
\ms

{\bf Remarks.} \begin{inparaenum}[(a)]
\item Let us have a closer look at the element $\b(u,v)$ in \eqref{stvr1}.
    By \eqref{3gra1} the possibilities for $\alp,\beta$ are $\alp \perp
    \beta$, $\alp \edge \beta$, $\alp \lra \beta$ and $\alp \lla \beta$ and
    by \eqref{rgjp4}, $Q_v u = 0$ unless $\beta \lra \alp$, and $Q_uv = 0$
    unless $\alp \lra \beta$. Therefore, by \eqref{sj1},
\begin{equation} \label{stvr3}
 \b(u,v) = \begin{cases} 1, & \text{if } \alp \perp \beta, \\
                   \dd {\xm(-v)} {\xp(u)}, & \text{if } \alp \edge \beta, \\
                   \xm(-Q_vu)\;\dd {\xm(-v)} {\xp(u)}, & \text{if }
                                           \alp \lra \beta, \\
                  \dd {\xm(-v)} {\xp(u)}\; \xp(-Q_uv), & \text{if } \alp \lla \beta.
 \end{cases} \end{equation}
In general, the factors $\xm(-Q_vu)$ and $\xp(-Q_uv)$ in the last two cases are $\ne 1$.
\sm

The reader may be puzzled by the definition of $\b(u,v)$: why not take
$``\b(u,v) = \dbl \xm(-v)\, , \allowbreak \xp(u)\dbr"$? We will give a
justification for this in \ref{thmB}. \ms

\item \label{stvr-b} We claim that \eqref{sj2} follows from \eqref{sj1}
in case $\alp \perp\beta$. Indeed, the left hand sides of the two equations
    \eqref{sj2} are $1$ because $\b(u,v)=1$, but also the right hand sides
    are $1$, since, say for $\sig=+$, we have  $\{u\, v\, z\} = 0$ by
    \eqref{rg2} and $Q_u Q_v z=0$ by \eqref{rgjp5} below, :
    %Because of \eqref{3gra4} this implies
\begin{equation} \label{rgjp5}
\hbox{$Q(V_\al\si)Q(V_\be\msi)V_\ga\si \ne 0 \implies \alp = \beta$ or $
\alp\lla \beta = \gam$ or $\alp \edge \beta \lla \gam \perp \alp$,}
\end{equation}
which can be shown by repeated application of \eqref{3gra2}. \sm

\item \label{stvr-c} Let $(V,\R)= (\MM_{IJ}(A), \R)$. Comparing the
    definition of $\St(\MM_{IJ}(A), \R)$ with the one in \ref{news}, it is
    clear that the first two relations coincide: \eqref{ej1} = \eqref{sj0}
    and \eqref{ej2} = \eqref{sj1}. We claim that the relations \eqref{ej3}
    coincide with the two relations in \eqref{sj2}. Indeed, since we do not
    have a relation $\alp \lla \beta$ in $\rma_N^I$, it follows from
    \eqref{stvr-b} and the assumption $\alp \ne \beta$ in \eqref{sj2} that
    we only need to consider the case $\alp \edge \beta$. Then the map
$V_\be^- \to    \St(V,\R)$, $v \mapsto \dd {\xp(u)}{\xm(v)}$ is homomorphism of groups,     whence, by \eqref{stvr3},
    \[ \b(u,v) = \dd {\xm(-v)}{\xp(u)} = \dd {\xp(u)}{\xm(-v)} \me =
      \dd {\xp(u)} {\xm(v)}.\] Thus  \eqref{ej3} = \eqref{sj2} for
       $\sig=+$ because $Q_uQ_x z = 0$ by \eqref{rgjp5}. The equality of
       the two relations for $\sig = -$ can be established in the
       same way.
\end{inparaenum}
\lv{%%%%%%%%%%%%%%%%%%%%%
Proof of \eqref{stvr-c} for $\sig = -$: By \eqref{stvr3} we have $\b(u,v)
\me = \dd {\xm(-v)}{\xp(u)}\me = \dd{\xm(v)}{\xp(u)}$, whence \eqref{sj2}
for $\sig = -$ becomes $\dd {\dd{\xm(v)} {\xp(u)}} {\xm(y)}  = \xm(- \{ v
\, u \, y\})$, which is \eqref{ej3} for $\sig=-$ after a change of
notation: $\dd{\dd{\xm(u)} {\xp(v)}} {\xm(z)}  = \xm(- \{ u \, v \, z\})$
}%%%%%%%%  end of lv
\ms

We can now state the generalization of part \eqref{kms-a} of the
Kervaire-Milnor-Steinberg Theorem~\ref{kms} in the setting of this section.

\subsection{Theorem A} \label{thmA}
{\em Let $(R, R_1)$ be a $3$-graded root system whose irreducible
components all have rank $\ge 5$, and let $V$ be a Jordan pair with a
fully idempotent root grading $\R$.
Then the Steinberg group $\St(V,\R)$ is centrally closed.} \ms

This theorem is one of the main results of \cite{LN}; its proof takes up
all of Chapter VI of \cite{LN}. It is shown there in greater generality.
First, it also true when $R$ has connected components of rank $4$, but not
of type $\rmd_4$. Moreover, it is not necessary to assume that the
root grading $\R$ is fully idempotent.
%with respect to a family of idempotents containing an
%idempotent in {\em every\/} $V_\al$, $\alp \in R_1$.
For the irreducible components of type $\rmb_I$ or $\rmc_I$, $|I|\ge 5$,
one only needs idempotents $e_\al \in V_\al$ in case  $\alp$ is a long root in
type $\rmb$ and $\alp$ a short root in type $\rmc$. This generality allows
us to consider groups defined in terms of hermitian matrices associated with form rings in the sense of \cite{bak}.

\subsection{Highlights of our approach.} \label{high}
The novel aspect of our approach is the consistent use of the theory of
$3$-graded root systems and Jordan pairs, which introduces new
methods in the theory of elementary and Steinberg groups. For example,
instead of first dealing with the case of finite root systems and then
taking a limit to get the stable (= infinite rank) case, we deal with both
cases at the same time. Moreover, our approach avoids having to deal with
concrete matrix realizations of the groups in question, as is traditionally
done, see e.g.\ \cite{bak} or \cite{hahn}. It allows for a concise
description of the defining relations, independent of the types of root
systems involved. Finally, as the discussion of the linear Steinberg group
in \ref{rw} -- \ref{mth} shows, we need fewer relations than in
previous work, for example no relations involving two roots in $R_0$.

With the exception of groups defined in terms of root systems of type
$\rme_8$, $\rmf_4$ and $\rmg_2$, which are not amenable to a Jordan approach, cf.\ \ref{3gra}\eqref{3gra-c},  our Theorem~\ref{thmA} covers all types of
Steinberg groups considered before. In addition, it also presents some new
types, e.g., for elementary orthogonal groups. A detailed comparison of our
Theorem~\ref{thmA} with previously known results is given in
\cite[27.11]{LN}. \ms

At this point it is natural to ask if there also exists a generalization of
part \eqref{kms-b} of Theorem~\ref{kms}, stating that the map $\wp \co
\St(A) \to \EL(A)$ is a universal central extension. While the group
$\St(V,\R)$ gives a satisfactory replacement for the linear Steinberg group
$\St(A)$, recasting the elementary linear group $\EL(A)$ in the framework
of Jordan pairs is limited to special Jordan pairs in the sense of
\ref{exjp}\eqref{defjp-ii}. While this can be done, see \cite{stab}, we
will instead replace the elementary group $\EL(A)$ by the projective
elementary group $\PE(V)$, see \ref{pev}, that can be defined for any
Jordan pair $V$. From the point of view of universal central extensions,
this is harmless since, as we will see in \ref{pvn}, the group $\PE(V)$ is
isomorphic to the central quotient $\PE(A) = \EL(A)/\scZ(\EL(A))$ and
universal central extensions of a group and its central quotients are
essentially the same by \ref{centex}\eqref{centex-d}.

%\newpage

\subsection{The Tits-Kantor-Koecher algebra and the projective elementary group
of a Jordan pair.} \label{pev}
%This subsection only requires the concept of a Jordan pair, but is
%otherwise independent of the previous development in this section. \sm
Let $V$ be a Jordan pair, defined over a commutative ring $k$ of scalars.
It is fundamental (and well-known) that $V$ gives rise to a $\ZZ$-graded
Lie $k$-algebra
\begin{equation}  \tkk(V)=\tkk(V)_{1} \oplus \tkk(V)_0
\oplus
\tkk(V)_{-1}, \label{intro2}
\end{equation}
introduced at about the same time by Tits, Kantor and Koecher in
\cite{Ti0,tits:alt,Ka1,kan:trans,kan:cert,koe:imbedI,Ko3} and called the
{\em Tits-Kantor-Koecher algebra of $V$\/}. Various versions of $\tkk(V)$
exist, but all agree that $\big( \tkk(V)_1, \tkk(V)_{-1}\big) = (V^+, V^-)$
as $k$-modules. For our purposes, the most appropriate choice for
$\tkk(V)_0$ is
\begin{align}\label{pev2}
 \tkk(V)_0 &= k\zeta + \Span_k \{ \del(x,y): (x,y) \in V\},
% \qquad \hbox{for $\zeta =(\Id_{V^+}, \Id_{V^-})$}
 \end{align}
where $\zeta =(\Id_{V^+}, \Id_{V^-})$ and  $\del (x,y) = (D(x,y), - D(y,x))\in \End(V^+) \times \End(V^-)$, defined by $D(x,y)z=\{x\, y\, z\}$. We let $\gl(V\si)$  be the Lie algebra defined by  $\End(V\si)$ with the commutator as the Lie product. By definition, the Lie product of $\tkk(V)$ is determined by the conditions that it be alternating, that $\tkk(V)_0$ be a subalgebra of the
Lie algebra $\gl(V^+) \times \gl(V^-)$ and that
\[
   [V\si, V\si] = 0, \qquad [D,z] = D_\sig (z), \qquad [x,y ] = - \del(x,y)
\]
for $D=(D_+, D_-) \in \tkk(V)_0$, $z\in V\si$ and $(x,y) \in V$.
%\comment Too many ``that'': \endcomment
It follows from the identity (JP15) in
\cite{jp},
\[ [ D(x,y), \, D(u,v)] = D(\{x\, y\, u\}, \, v) -  D(u, \, \{y\, x\, v\}).
\]
that $\tkk(V)_0$ is indeed a subalgebra. As a $k$-Lie
algebra, $\tkk(V)$ is generated by $\zeta$, $V^+$ and $V^-$, and it
has trivial centre.
%\comment
%``Its
%centre vanishes because $\zeta \in \tkk(V)$ (this is one of the reasons for
%our choice of $\tkk(V)_0$).''
%
%So wie er dasteht, ist dieser Satz raetselhaft.
%\endcomment
\sm

For a Jordan pair $V$ with a fully idempotent root grading $\R$ a
description of the derived algebra $[\tkk(V), \tkk(V)]$ is given in
\cite{n:3g}. The Tits-Kantor-Koecher algebra of a special Jordan pair is
described in \cite[\S2]{stab}. We will work out $\tkk(V)$ for a rectangular matrix pair in \ref{rect-pvex}. \sm

An automorphism $f$ of $V$ gives rise to an automorphism $\tkk(f)$ of
$\tkk$, defined by
\[ x \oplus D  \oplus y \quad \mapsto \quad
          f_+(x) \oplus (f\circ D\circ f\me)  \oplus f_-(y).\] The map $f
\mapsto \tkk(f)$ is an embedding of the automorphism group $\Aut(V)$ of $V$
into the automorphism group of $\tkk(V)$. \ms

Any $(x,y)\in V$ gives rise to automorphisms $\exp_+(x)$ and $\exp_-(y)$ of
$\tkk(V)$, defined in terms of the decomposition \eqref{intro2} by the
formal $3\times 3$-matrices
\begin{equation} \label{pev3}
\exp_+(x)=\begin{pmatrix} 1 & \ad x & Q_x
\\ 0 & 1 & \ad x \\ 0&0&1 \end{pmatrix} , \qquad \exp_-(y)=\begin{pmatrix}
1&0&0\\ \ad y&1&0\\ Q_y&\ad y&1 \end{pmatrix}.
\end{equation} The map
$\exp_\sig$, $\sig = \pm$, is an  injective homomorphism from the abelian
group $(V\si, +)$ to the automorphism group
%$\Aut\big(\tkk(V)\big)$
of $\tkk(V)$, whose image is denoted $U\si$. The {\it projective elementary
group of $V$\/} is the subgroup $\PE(V)$ of $\Aut\big(\tkk(V)\big)$
generated by $U^+ \cup U^-$, introduced in  \cite{stab} and studied further
in \cite[\S7, \S8]{LN}.
\ms

We have now explained all the concepts used in the generalization of part
\eqref{kms-b} of the Kervaire-Milnor-Steinberg Theorem~\ref{kms}.

\subsection{Theorem B}\label{thmB}{\em
Let $(R, R_1)$ be a $3$-graded root system and let $V$ be a Jordan pair
with a root grading $\R=(V_\al)_{\alp \in R_1}$. \ms

\begin{inparaenum}[\rm (a)]
  \item \label{thmB-a} There exists a group homomorphism $\pi \co \St(V,\R)
      \to \PE(V)$, uniquely determined by
      \[ \pi\big( \xs(u) \big) = \exp_\sig (u), \qquad (u\in V\si). \]

\item \label{thmB-b} If all irreducible components of $R$ have infinite
    rank and $\R$ is fully idempotent with respect to a family
    $(e_\al)_{\alp \in R_1}$, the homomorphism $\pi$ is a universal central
    extension.
\end{inparaenum}}
\ms

Theorem B is established in \cite{LN}. Part~\eqref{thmB-a} follows from
\cite[Cor. 21.12]{LN}. By Fact~\ref{centex}\eqref{centex-c} and
Theorem~\ref{thmA}, the proof of \eqref{thmB-b} boils down to showing that
$\Ker \pi$ is central, which we do in \cite[Cor.~27.6]{LN}. As for
Theorem~\ref{thmA}, it is not necessary to assume that $\R$ is fully idempotent.
\sm

In the setting of \eqref{thmB-a} let $(u,v) \in V_\al^+ \times V_\be^-$
with $\alp\ne \beta$ and let $\b(u,v)$ be the element of $\St(V,\R)$
defined in \eqref{stvr1}. Then $\pi(\b(u,v)) = \tkk(f)$ for some $f\in
\Aut(V)$ (for the experts: $f$ is the inner automorphism $(B(u,v),
B(v,u)\me)$ of \cite[3.9]{jp}). That $\pi(\b(u,v)) \in
\tkk\big(\Aut(V)\big) \subset \Aut(\tkk(V)) $ is the motivation for the perhaps surprising definition of $\b(u,v)$.
%%en folgenden Absatz wuerde ich streichen:
%
%``In the setting of \eqref{thmB-a} let $(u,v) \in V_\al^+ \times V_\be^-$
%with $\alp\ne \beta$ and let $\b(u,v)$ be the element of $\St(V,\R)$
%defined in \eqref{stvr1}. Then $\pi(\b(u,v)) = \tkk(f)$ for some $f\in
%\Aut(V)$ (for the experts: $f$ is the inner automorphism $(B(u,v),
%B(v,u)\me)$ of \cite[3.9]{jp}). That $\pi(\b(u,v)) \in
%\tkk\big(\Aut(V)\big)$ %is the motivation for the perhaps surprising %definition of $\b(u,v)$.''
%\medskip
%$\pi(\b(u,v)) \in \tkk\big(\Aut(V)\big)$ sieht sehr merkwuerdig aus,
%ist hier $\tkk$ als Funktor zu verstehen?
%\endcomment
\sm

We finish this section by describing $\tkk(V)$ and $\PE(V)$ for
$V=\MM_{IJ}(A)$.

\subsection{The Tits-Kantor-Koecher algebra of a rectangular matrix
pair.}  \label{rect-pvex} Let $V=\MM_{IJ}(A)= \big(\Mat_{IJ}(A),
\Mat_{JI}(A)\big)$ be the rectangular matrix pair of
\ref{exjp}\eqref{defjp-iir}. In this subsection we present a model for the
Tits-Kantor-Koecher algebra $\tkk = \tkk(V)$ in terms of elementary
matrices which will be used in \ref{pvn} to link the elementary group of $V$ and the abstractly defined group $\PE(V)$. \sm

Let ${\bf 1}_I = \diag(1_A, \ldots )$ be the diagonal matrix of size
$I\times I$, define ${\bf 1}_J$ analogously and let $\frA$ be the unital
associative $k$-algebra
\[
   \frA = \frA(V) = \begin{pmatrix}
     k \, {\bf 1}_I + \Mat_I(A) & \Mat_{IJ}(A) \\ \Mat_{JI}(A) & k \,{\bf 1}_J +
       \Mat_{J}(A)  \end{pmatrix}
   =\begin{pmatrix}
      \Mat_I(A)_{\rm ex} & \Mat_{IJ}(A) \\ \Mat_{JI}(A) & \Mat_{J}(A)_{\rm ex}
       \end{pmatrix}
\]
whose  operations are given by matrix addition and matrix multiplication.
In particular,
\[
   e_1 = \begin{pmatrix} {\bf 1}_I & 0 \\ 0 & 0 \end{pmatrix} \quad \qquad \hbox{and} \quad  \qquad
   e_2 = \begin{pmatrix} 0 & 0 \\ 0 & {\bf 1}_J \end{pmatrix}
\]
are orthogonal idempotents of $\frA$.
We consider  $\frA$ rather than its subalgebra $\Mat_N(A)_{\rm ex}$, $N= I \dotcup J$, since this will
allow us to model the element $\zeta$ of \eqref{pev2}.

The Peirce decomposition of $\frA$ with respect to the idempotent $e_1$ is
\begin{align*}
 \frA_{11} &= \begin{pmatrix} \Mat_I(A)_{\rm ex} & 0 \\ 0 & 0 \end{pmatrix}, &
 \frA_{10} &= \begin{pmatrix} 0 & \Mat_{IJ}(A) \\ 0 & 0 \end{pmatrix}, \\
 \frA_{01} &= \begin{pmatrix} 0 & 0 \\ \Mat_{JI}(A)  & 0 \end{pmatrix}, &
\frA_{00} & \begin{pmatrix} 0 & 0 \\ 0 & \Mat_J(A)_{\rm ex} \end{pmatrix}.
\end{align*}
Let $\frA^{(-)}$ be the Lie algebra associated with $\frA$. Thus,
$\frA^{(-)}$ is defined on the $k$-module underlying $\frA$ and its Lie
algebra product is $[x,y] = xy-yx$ for $x,y\in \frA$. The Lie algebra
$\frA^{(-)}$ is $\ZZ$-graded, $\frA^{(-)} = \bigoplus_{n\in \ZZ}
\frA^{(-)}_n$ with
\[
    \frA^{(-)}_1 = \frA_{10}, \quad \frA^{(-)}_0 = \frA_{11}
            \oplus \frA_{00}, \quad \frA^{(-)}_{-1} = \frA_{01}
\]
and $\frA^{(-)}_n = 0$ for $n\notin \{1, 0 , -1\}$.  We define $\fre=
\fre(V)$ as the subalgebra of $\frA^-$ generated by $e_1$, $e_2$ and
\[ \fre_1 = \begin{pmatrix} 0 & \Mat_{IJ}(A) \\ 0 & 0 \end{pmatrix} = \frA^{(-)}_1
\quad \hbox{and} \quad \fre_{-1}  = \begin{pmatrix} 0 & 0 \\ \Mat_{JI}(A) & 0
   \end{pmatrix} = \frA^{(-)}_{-1}.
\]
Put $\fre_0 = k \, e_1 +  k \, e_2 + [\fre_1, \, \fre_{-1}]$ and
$\fre_i = 0$ for $i \notin \{-1, 0,1\}$.  Then $\fre = \bigoplus_{i \in \ZZ}
\fre_i$ is a $\ZZ$-graded Lie algebra.

We now relate \ignore the Lie algebra \endignore $\fre$ to the Tits-Kantor-Koecher algebra
$\tkk = \tkk(V)$ of $V$. First, for $\sfa = \big( \begin{smallmatrix} a & 0 \\
0 &d \end{smallmatrix}\big) \in \fre_0$  define $\Del(\sfa)  = (\Del(\sfa)_+, \,
\Del(\sfa)_-) \in \End_k(V^+) \times \End_k(V^-)$ by
\[
   \Del(\sfa )_+ (u) = au - ud, \qquad \Del(\sfa)_- (v) = dv-va,
\]
so that
\begin{equation} \label{rect-pvex-7}
 \Big[ \sfa, \pma 0bc0 \Big] =
      \pma 0 {au-ud}{dv-va}0 = \pma 0 {\Del_+(\sfa)(u)} {\Del_-(\sfa)(v)} 0.
\end{equation}
We claim: {\em the map
\[
   \Psi \co \fre \to \tkk, \qquad \begin{pmatrix} a & b \\ c & d \end{pmatrix}
   \mapsto b \oplus  \Del(a,d) \oplus (-c)
\]
is a surjective Lie algebra homomorphism  whose kernel is $\frz(e)$, the
centre of $\fre$, and thus induces an isomorphism
\begin{equation} \label{rect-pvex1}
  \fre / \frz(\fre) \cong \tkk
\end{equation}
of Lie algebras\/} (\cite[2.6]{stab}, \cite[7.2]{LN}). Indeed, $\Psi$  is
surjective since $\Del(e_1) = \zeta = - \Del(e_2)$ and for $(x,y) \in V$
\begin{equation} \label{rect-pvex-8}
 \Del\big( \big[ \big(\begin{smallmatrix} 0 & x \\ 0 & 0 \end{smallmatrix}\big),
\, \big( \begin{smallmatrix} 0 & 0 \\ y & 0 \end{smallmatrix}\big) \big] \big)
 = \Del \big( \begin{smallmatrix} xy & 0 \\ 0 & -yx \end{smallmatrix} \big)
 = \del(x,y)
\end{equation}
by \eqref{defjo-ii1}.  To see that $\Ker \Psi = \frz(\fre)$, observe for
$\sfm = \sma abcd \in \frA$ that
\begin{equation}\label{rect-pvex-5}
%\{ x\in \frA : [x,e_1]  = 0 \} =  \frA_{11} \oplus \frA_{00}
%  = \{ x\in \frA : [x,e_2]  = 0 \},
[\sfm , e_1] = 0 \iff b=0=c \iff [\sfm, e_2]=0,
\end{equation}
whence by \eqref{rect-pvex-7}, %for $\sfm = \sma abcd \in \frA$ we have
\begin{align*}
  \sfm \in \Ker \Psi \quad &\iff \quad b=0=c,\; \De(a,d) = 0,\; \sma a00d \in \fre_0, \\
 &\iff \quad [\sfm, e_1]=0 = [\sfm,e_2],\; [\sfm, \,\fre_1]=
    0=[\sfm, \, \fre_{-1}], \; \sfm \in \fre, \\
& \iff \quad \sfm \in \frz(\fre)
\end{align*}
because $e_1, e_2, \fre_1$ and $\fre_{-1}$ generate $\fre$ as Lie algebra.
Finally, since both $\fre$ and $\tkk$ are $\ZZ$-graded and $\Psi$ preserves
this grading, $\Psi$ is a Lie algebra homomorphism as soon as $\Psi$
preserves products of type $[\sfx_i, \sfy_j]$ for $(i, j)= (0, \pm 1), (1,
-1)$ and $(0,0)$. For $(0, \pm 1)$ and $(1,-1)$ this follows from
\eqref{rect-pvex-7} and \eqref{rect-pvex-8} respectively; the case $(0,0)$
is left to the reader. \ms

In the remainder of this subsection we will give a more precise description
of $\fre$ and its centre, see \eqref{rect-pvex-0} and \eqref{rect-pvex-9}.
Let $[A,A]= \Span \{ab-ba : a,b\in A\}$, the derived algebra of the Lie
algebra $A^-$, and let
$$
\lsl_N(A) := \{ x=(x_{k l}) \in \Mat_N(A) \co  \textstyle \sum_{n\in N} x_{nn} \in [A,A]\}.
$$
From $[aE_{kl}, b E_{rs}] = ab \del_{lr}
E_{ks} - ba \del_{ks} E_{rl}$ one then gets
\begin{equation} \label{rect-pvex-0}
 \begin{split} \fre_1 \oplus [\fre_1, \fre_{-1}] \oplus \fre_{-1} &= \lsl_N(A),
                 \qquad \hbox{whence} \\
           k \, e_1 + k \, e_2 + \lsl_N(A) & = \fre. \end{split}
\end{equation}
\lv{%%%%%%%%%%%%%%%%%%%%%%%%%%%
Description of $\fre$: $[a E_{ij}, bE_{jl}] = ab E_{il} - \del_{il} ba
E_{ji}$ and $[aE_{ij}, bE_{mi}] = \del_{jm} ab E_{ii} - ba E_{mj}$ whence
all off-diagonal elements of $\lsl_N(A)$ lie in $\fre$, and so do all
$abE_{ii} - baE_{jj}$, where now $i,j\in N$, $i\ne j$. One then concludes
because the latter elements span the diagonal of $\lsl_N(A)$: take $b=1$ to
get $aE_{ii} - a E_{jj}$, whence $(ab E_{ii} - ba E_{jj}) - (rba E_{ii} -
ba E_{jj}) = [a,b] E_{ii} \in \fre$.
}%%%%%%%%%%%  end of lv %%%%%%%%%%%%%%%%%%%%%%%%%%%%%%%%%
The description of the centre $\frz(\fre)$ depends on the cardinality of
$N$ because
\[ %\tag{*}
\frA^{(-)}_0 \cap A {\bf 1}_N= \begin{cases} A \,{\bf 1}_N & \text{if }
                                             |N|< \infty, \\
                  k \,  {\bf 1}_N & \text{if } |N|=\infty. \end{cases}
\]
\lv{%%%%%%%%%%%%%%%%%%%%%%%%%%%%%%%%%%%%%%%%%%%%%%%%%%%%%%%%%%%%%%%%%%%%%%
We have $\frA_{11} \oplus \frA_{00} = (k e_1 + \Mat_I(A)) \oplus (k e_2 +
\Mat_J(A))$. If $|N|< \infty$ then both $I$ and $J$ are finite, so $e_1 \in
\Mat_I(A)$ and $e_2 \in \Mat_J(A)$ and $a{\bf 1}_N = a e_1 + a e_2 \in
\frA_{11} \oplus \frA_{00}$. If $N$ is infinite, then $I$ or $J$ is
infinite, say $I$ is infinite. An element $a {\bf 1}_N$ lying in $\frA_{11}
\oplus \frA_{00}$ has the form $a{\bf 1}_N = (s_1 e_1 + x') \oplus (s_2 e_2
+ x'')$ with $s_i \in k$, $x'\in \Mat_I(A)$ and $x'' \in \Mat_J(A)$, in
particular $(a - s_1){\bf 1}_I = x'\in \Mat_I(A)$ has only finitely many
non-zero coefficients, forcing $a=s_1\in k$. Conversely, $k {\bf 1}_N = k
(e_1 + e_2)\in \frA_0^{(-)}$ is always true.}%%%%%%%%%%%%%%  lv ends
Denoting by $\rmZ(A)= \{ z\in A : [z,A]=0\}$ the centre of $A$ (= centre of
$A^{(-)}$), a straightforward calculation shows
\begin{equation}  \label{rect-pvex-6}
 \big\{ x\in \frA^{(-)}_0 : [x, \fre_1] = 0 = [x, \fre_{-1}]\big\}
   = \frA^{(-)}_0 \cap \rmZ(A){\bf 1}_N=  \begin{cases}
     \rmZ(A) {\bf 1}_N, & |N|< \infty, \\
       k {\bf 1}_N, & |N|= \infty .  \end{cases}
\end{equation}
\lv{%%%%%%%%%%%%%%%%  lv starts %%%%%%%%%%%%%%%%%%%%%%%%%%%%%%%%%%%
Proof of \eqref{rect-pvex-6}: Assume $x=(x_{ij})\in \frA_{11} \oplus
\frA_{00}$  commutes with $\fre_1 + \fre_{-1}$, whence by
\eqref{rect-pvex-0} with $\lsl_N(A)$. For $k,l\in N$, $k\ne l$ we then get
$0 = [x, E_{kl}] = \sum_{r\in \NN} x_{rk} E_{rl} - \sum_{s\in N} x_{ls}
E_{ks}$. For $E_{rl} \ne E_{ks}$, i.e., $(rs) \ne (kl)$, we obtain $x_{rk}
= 0$ for $r\ne k$, and for $(rl)=(ks)$ we get $x_{kk} = x_{ll}$. Thus
$x=z{\bf 1}_N$ for some $z\in A$. For arbitrary $a\in A$ and again $k \ne
l$ we get $0 = [z{\bf 1}_N, a E_{kl}] = (za - az)E_{kl}$, so that $z\in
\rmZ(A)$ follows. Hence $x\in \frA^{(-)}_0 \cap \rmZ(A) {\bf 1}_N$. The
other inclusion is clear. The second equality with the case distinction is
immediate from (*).
}%%%%%%%%%%%%%%%% lv ends
Since $\fre$ is generated by $e_1$, $e_2$, $\fre_1$ and $\fre_{-1}$,
\eqref{rect-pvex-5} and \eqref{rect-pvex-6} imply
\begin{equation}\label{rect-pvex-9}
\frz(\fre) = \fre_0 \cap (\rmZ(A) \, {\bf 1}_N)
  =  \begin{cases}
     \fre_0 \cap (\rmZ(A) {\bf 1}_N), & |N|< \infty, \\
       k {\bf 1}_N, & |N|= \infty .  \end{cases}.
\end{equation}

For example, if $A=k$ is a field of characteristic $0$ and $|N|=n$ is
finite, we get $\fre = \gl_n(k)$, $\frz(\fre) = k {\bf 1}_n$ and $\tkk
\cong \lsl_n(k)$.

\subsection{The projective elementary group of a rectangular matrix pair}
\label{pvn} We use the notation of \ref{rect-pvex} and let
$V=\MM_{IJ}(A)$. The goal of this  subsection is to show that the group $\PE(V)$ is isomorphic to a central quotient of the elementary group $\EL(A)$. We put
\[ \EL(V) = \EL_N(A) \]
and call it the {\em elementary group of $V$\/}.  Since by \ref{exc} the
group $\EL_N(A)$ is generated by $\e_+(V^+) \cup \e_-(V^-)$ this agrees
with the definition of the elementary group of an arbitrary special Jordan
pair in \cite[\S2]{stab} or \cite[6.2]{LN}. We will identify the
Tits-Kantor-Koecher algebra $\tkk(V)= \tkk$ with $\fre/\frz(\fre)$ via the
isomorphism \eqref{rect-pvex1} induced by $\Psi \co \fre \to \tkk$. \sm

Any $g\in \frA\ti$ gives rise to an automorphism $\Ad g$ of $\frA^-$
defined by $(\Ad g)(x) = gx g\me$. If $\Ad g$ stabilizes the subalgebra
$\fre$ of $\frA^-$, it also stabilizes $\frz(\fre)$, and therefore descends
to an automorphism $\uAd (g)$ of $\fre/\frz(\fre)= \tkk$ satisfying
$\Psi\circ (\Ad g |_\fre) = (\uAd g) \circ \Psi$. The map
\[ \uAd \co \{ g \in \frA\ti : (\Ad g)(\fre) = \fre\} \to \Aut (\tkk), \qquad
g \mapsto \uAd g
\]
is a group homomorphism. For $g=\e_+(x) = \big(
\begin{smallmatrix} 1 & x \\ 0 &1 \end{smallmatrix}\big) \in \frA\ti$ as in \eqref{exc0},
the automorphism $\Ad \e_+(x)$ acts as follows:
\begin{align*}
  \big( \Ad \e_+(x) \big)\, \pma 0  b  0 0  &= \pma 0  b 0 0 , \\
 \big( \Ad  \e_+(x) \big)\, \pma a  0 0 d   &= \pma a {-ax + xd} 0 d
  = \pma  a 0 0 d  +  \Big[ \pma 0 x 0 0  \, , \, \pma a 0 0 d \Big], \\
 \big( \Ad \e_+(x)\big)\, \pma 0 0 {-c} 0 &= \pma {-xc} {xcx} {-c} {cx} \\ & =
    \pma 00{-c} 0 + \Big[\pma 0 x  0 0 \, , \, \pma 0 0 {-c} 0 \Big]
         + \pma 0 {Q_xc} 0 0.
 \end{align*}
These equations show that the automorphism $\Ad \e_+(x)$ stabilizes $\fre$
and,  by comparison with \eqref{pev3}, that $\uAd \e_+(x) = \exp_+(x)$. One
proves in the same way that $\Ad \e_-(y)$, $y\in V^-$, stabilizes $\fre$
and that  $\uAd \e_-(y) = \exp_-(y)$. Since
$\PE(V)$ is generated by $\exp_+(V^+) \cup \exp_-(V^-)$, the
homomorphism $\uAd$ restricts to a surjective group homomorphism
\[ \uAd_{\EL} \co \EL(V) \to \PE(V), \quad g \mapsto \uAd g.\]
We claim that its kernel is the centre $\rmZ(\EL(V))$ of $\EL(V)$:
\begin{align} \label{pvn-1}
 \rmZ(\EL(V)) = \Ker \Ad |_{\EL(V)} &= \Ker \uAd_{\EL}, \qquad \hbox{whence}
 \\   \EL(V)/ \rmZ(\EL(V)) &\cong \PE(V).
\end{align}
Proof of \eqref{pvn-1}: Clearly, $\rmZ(\EL(V)) = \Ker \Ad |_{\EL(V)}
\subset \Ker \uAd_{\EL}$, so it remains to show that $g\in \Ker \uAd_{\EL}$
is central in $\EL(V)$. Let $g=\sma abcd$ and $g\me = \sma
{a'}{b'}{c'}{d'}$. Then
\begin{align}
  g \pma {{\bf 1}_I}000 g\me &= \pma {aa'}{ab'} {ca'} {cb'},\label{pvn-2}
   \\
   g \pma 000{{\bf 1_J}}  g\me &= \pma {bc'}{bd'} {dc'} {dd'} ,  \label{pvn2a}\\
  g\me \pma 0x00 g &= \pma {a'xc}{a'xd}{c'xc}{c'xd}, \label{pvn-3} \\
  g\me \pma 00y0 g &= \pma {b'ya}{b'yb}{d'ya}{d'yp}. \label{pvn-4}
\end{align}
Since %\begin{equation*}  \label{pvn-2}
$(\Ad g)\, (\sfm) \equiv \sfm  \equiv (\Ad g \me)(\sfm) \mod \frz(\fre)$ %\end{equation*}
for all $\sfm \in \fre$ and since $\frz(\fre)$ is diagonal by
\eqref{rect-pvex-9}, it follows from \eqref{pvn-2},  \eqref{pvn2a} and
\eqref{pvn-4} that
\[
ab'= 0 = ca' = bd' ,  \qquad d'ya= y \hbox{ for $y \in
V^-$}.
\]
 Applied to $c\in V^-$,  this proves $cb' = (d'ca)\cdot b' = d'c
\cdot ab' = 0$ and then $bc' = b \cdot ( d'c'a) = bd' \cdot c'a = 0$. From
${\bf 1}_N = g g\me$ we obtain
\[ \pma{ {\bf 1}_I}00{{\bf 1}_J} = \pma{ aa' + bc'} *** = \pma{aa'}***.\]
Together with the already established equations this shows, using
\eqref{pvn-2}, that $(\Ad g)(e_1) = e_1$. Because $g\in \frA$, we get
$b=0=c$ from \eqref{rect-pvex-5}. Thus also $b'=0=c'$, so that
\eqref{pvn-3} and \eqref{pvn-4} prove that $\Ad g$ fixes $\fre_{\pm 1}$.
Since $\e_\pm(V^\pm) = {\bf 1}_N + \fre_{\pm 1}$, we see that $\Ad g$ fixes
the generators $\e_{\pm} (V^\pm)$ of $\EL(V)$, i.e. $g$ is central. \qed
\sm

A similar result holds for any special Jordan pair $V$: there always
exists a surjective group homomorphism from the elementary group $\EL(V)$
(which we have not defined) onto the projective elementary group $\PE(V)$,
whose kernel is central, but not necessarily the centre of $\EL(V)$, see
\cite[Thm.~2.8]{stab}.

\section{Some open problems}\label{sec:open}
We describe some open problems for Steinberg and projective elementary
groups of Jordan pairs. Our list is very much limited by the author's taste
and knowledge. This section requires some expertise in Jordan pairs.

\subsection{The normal subgroup structure of $\PE(V)$}\label{nsp} The problem
is quite easily stated: {\it  Given a Jordan pair $V$, describe all normal
subgroups of $\PE(V)$.} As stated, this may be too general. We therefore discuss some special cases. \sm

\begin{inparaenum}[(a)]
\item \label{nsp-a} In view of the results of \cite{simp} it is natural to
    ask: when is $\PE(V)$ a perfect group, when is it simple? Indeed,
    \cite[Thm.~2.6]{simp} says that, for a nondegenerate Jordan pair $V$
    with dcc on principal ideals, $\PE(V)$ is a perfect group if and only if
    $V$ has no simple factors isomorphic to $(\FF_2, \FF_2)$, $(\FF_3,
    \FF_3)$ or $\big(\rmH_2(\FF_2), \rmH_2 (\FF_2)\big)$. Here $\FF_q$ is
    the field with $q$ elements. Also, by \cite[Thm.~2.8]{simp}, $\PE(V)$
    is a simple (abstract) group if and only if $V$ is simple and not
    isomorphic to $(\FF_2, \FF_2)$, $(\FF_3, \FF_3)$ or
    $\big(\rmH_2(\FF_2), \rmH_2 (\FF_2)\big)$ (\cite[Thm.~2.8]{simp}). That
    the exceptional cases have to be excluded in these two theorems is
    evident from the isomorphisms $\PE(\FF_2, \FF_2) \cong \frS_3$ (the
    symmetric group on three letters), $\PE(\FF_3, \FF_3) \cong \frA_4$ (the
    alternating group on four letters), and $\PE\big( \rmH_2(\FF_2),
    \rmH_3(\FF_2)\big) \cong \frS_6$. Thus, the problem is to find out if
    these two theorems of \cite{simp} hold for more general Jordan pairs.
    It follows from \eqref{nsp-b} that is natural to assume simplicity of
$V$ for the second theorem. \sm

\item\label{nsp-b}  Every ideal $I$ of the Jordan pair $V$ gives rise to a
    normal subgroup of $\PE(V)$. Indeed, one can show (\cite[7.5]{LN}) that
    the canonical map $\can \co V \to V/I$ induces a surjective group
    homomorphism $\PE(\can) \co \PE(V) \to \PE(V/I)$. We let $\PE(V,I)$ be
    its kernel:
$$ \xymatrix@C=35pt{1 \ar[r] &\PE(V,I) \ar[r] &\PE(V) \ar[r]^{\PE(\can)} & \PE(V/I) \ar[r]& 1}
$$
Problem: describe $\PE(V,I)$ by generators and relations. For elementary
linear groups over rings this is a standard result, see for example
\cite[Lemma~3]{hazvav}. The paper \cite{ck:ann} shows that even in case
$\SL_2(A)$ one needs methods from Jordan algebras.
\end{inparaenum}

\subsection{Central closedness of $\St(V,\R)$ in low ranks}\label{cenco}
We have excluded low rank cases in Theorem~\ref{thmA} for the simple reason that it is not true without further assumptions in low ranks. We
discuss $2 \le \rank R \le  4$ in (a) and $\rank R=1$ in (b). \sm

(a) One knows (\cite[27.11]{LN}) that $\St(V,\R)$ is a classical linear or
unitary Steinberg group. Let us first consider the case that $V$ is defined
over a field $F$ and that $\dim_F V_\al = 1$ for all $\alp \in R_1$. Then
\cite[Thm.~1.1]{st3} applies and yields that $\St(V,\R)$ is not centrally
closed if and only if $(R, R_1)$ and $F$ satisfy one of the following
conditions.
\begin{equation} \vcenter{\offinterlineskip
\def\strut{\vrule height 14pt depth 6pt width 0pt}%
\def\bigstrut{\vrule height 18pt depth  12pt width 0pt}%
\everycr={\noalign{\hrule}}%
\halign{% Musterzeile:
\strut\vrule\hfil$\quad #\quad$\hfil\vrule\vrule % 1. Spalte
&&\hfil$\;\;#\;\;$\hfil\vrule% der && bewirkt, dass alle folgenden
                                       % Spalten genauso behandelt werden.
\cr                                    % Ende der Musterzeile
%\Gam & \rmk_2 & \rmk_3 \hbox{ or } \rmk_2\pten \rmk_2 & \xo_1 = \xt_2
%& \xo_2 & \xt_3& \OO_3= \TT_4 \cr
(R, R_1)& \rmaf_2^1 &  \rmaf_3^1 \hbox{ or } \rmaf_3^2 & \rmc_2\her & \rmb_3\qf &
  \rmc_3\her &  \rmd_4\alt
\cr
|F| & 2,4 & 2 &  2 & 3 & 2 & 2 \cr } }
\label{cenclo2}
\end{equation}
%(In case $\rmb_2\qf$ the group $\St(V,\R)$ is not perfect % by \no{idealpev}(a)
%and thus not centrally closed; this case is missing in \cite{st3}.)
In this table we use the abbreviation $\rmaf_2^1= \rmaf_N^I$ for $|N|=2$,
$|I|=1$ and analogously for $\rmaf_3^1$, \dots, $\rmd_4\alt$. The cases
$R=\rmaf_4, \rmb_4, \rmc_4$ do not appear in the table because in these
cases $\St(V,\R)$ is centrally closed, as mentioned in \ref{thmA}. \sm

Still assuming that $V$ is defined over a field $F$, it is natural to
replace the assumption $\dim_F V_\al = 1$ by the requirement that the
fully idempotent root grading $\R$ of $V$ is a {\it division grading\/}
in the sense that all root spaces $V_\al$ are Jordan division pairs,
%\new over $F$??\endnew, \new
which means that for every non-zero $x\in V_\al\si$  the endomorphism
$Q(x)|V_\al\msi$ is invertible.
%Using \cite[Th.~5.2]{comp}, it is not hard to show that this last assumption is
%equivalent to assuming that $V$ simple and complemented, equivalently,
%simple Artinian and nondegenerate.
% to the general case of a division graded $(V,\R)$, it is no
%harm to assume that $V$ is defined over a field $F$ (\cite[1.18]{jp}).
Preliminary investigations lead us to conjecture:
%{\em If\/ $\St(V,\R)$ is not centrally closed, then $\dim_F V_\al= 1$ for all $\alp \in R_1$ and $F$
%satisfies the restrictions of table \eqref{cenclo2}. }
\begin{enumerate}
\item[(C)] {\em If\/ $\St(V,\R)$ is not centrally closed then $\dim_F     V_\al= 1$ for all $\alp \in R_1$ and $F$ satisfies the restrictions
    of table \eqref{cenclo2}. }
\end{enumerate}
%\comments{ (E, 2018-06-19) Das stimmt im Fall $\rmaf_2$ -- siehe mein Manuskript
%Central closure rk(V)=1 l(v)= 2 (scan).}
\sm

(b) $R=\rmaf_1$: As in (a) we assume that $\R$ is a division grading, i.e.,
$V$ is a division pair and is therefore isomorphic to the Jordan
    pair $(J,J)$ of a division Jordan algebra $J$. By \cite[9.13]{LN} this
    is equivalent to $\PE(V)$ being a rank one group in the sense of
    \cite{lrone}. % a triple $(G, U^+, U^-)$ consisting of a non-trivial
    %group $G$ and two subgroups $U^+$ and $U^-$ satisfying the following
    %conditions: $G$ is generated by $U^+$ and $U^-$,  $U^+ \cap U^- =
    %\{1\}$, and for all $\sig \in \{+,-\}$ and all $x \in \dot U\si := U\si
    %\setminus \{1\}$ there exists $y \in U\msi$ such that $x\, U\msi x\1 =
    %y \, U\si y\1$.
    Since the grading is trivial, $\St(V,\R)$ is the free product of the
    abelian groups $V^+$ and $V^-$, which is not perfect in general, a
    necessary condition for a group to be centrally closed
    (\ref{centex}\eqref{centex-a}). Following the example of Chevalley
    groups \cite{st2}, it seems more promising to consider the group
    $\St(J)$ defined by the following presentation:

\begin{itemize}\item generators $\xs(a)$, $a\in J$, $\sig=\pm$ and, putting
       \[ \rmw_b = \xm(b\me)\,\xp(b) \, \xm(b\me)\] for $0\ne b\in J$,

  \item relations \begin{align*}   \xs(a+b) &= \xs(a)\, \xs(b) \hbox{ for
      $a,b\in J$ and} \\
  \rmw_b \,\xm(a) \, \rmw_b\me &= \xp\big( U(b) a) \hbox{ for all $a\in
J$ and all $0\ne b\in J$.}
\end{align*}
\end{itemize}
We remark that $\St(J)$ is the Steinberg group $\St(V, \scS)$ of
    \cite[13.1]{LN}, where $\scS$ is the set of all non-zero idempotents of
    $V$. By \cite[13.6]{LN}, $\St(J)$ is the classical Steinberg group $\St(A)$ in  case $V=(A,A)$ and $A$ an associative division algebra. \sm

To motivate our conjecture in this case, let us first consider the special
case $J=\FF_q$. Since by \ref{nsp}\eqref{nsp-a} the group $\PE(V)$ is not
perfect in case $J=\FF_q$, $q=2,3$, these cases have to be excluded.
Moreover, by \cite[Th.~1.1]{st3}, $\St(J)$ is not centrally closed in case
$V =(\FF_q, \FF_q)$ and $q\in \{4,9\}$, but these values of $q$ are the only exceptions for $V=(F, F)$, $F$ a field. This leads us to ask:
\begin{itemize}
\item[(Q)] {\it Is\/ $\St(J)$ centrally closed whenever $J\ne \FF_q$ with
    $q\in \{2,3,4, 9\}$ ?}
\end{itemize}
There exists an example of an associative unital $\FF_5$-algebra $A$ for which $\St(A)$ is not centrally closed \cite[Ex.~4]{stro}, but $A$ is not a division algebra.

\subsection{Centrality of $\Ker(\pi)$.} Let $\pi \co \St(V,\R) \to \PE(V)$ be the
homomorphism of Theorem~\ref{thmB}\eqref{thmB-a}. For simplicity, let us
assume that $R$ is irreducible. If $R$ has infinite rank, part
\eqref{thmB-b} of \ref{thmB} says that $\pi$ is a universal central
extension. The problem here is: {\em find sufficient conditions for $\Ker (\pi)$ to be central if $R$ has finite rank.}

Some special cases are known. For example, if $V$
is split in the sense of \cite{n:poly}, centrality of $\Ker(\pi)$ is established in \cite{vdK-ap}, \cite{lavrenov}, \cite{lav-sin} and \cite{sinch} for
$\rank  \ge 3$. The quoted papers all use the same method, pioneered by
\cite{vdK-ap}, namely a ``basis-free presentation of $\St(V,\R)$''. {\it
Can the method of\/ \cite{vdK-ap} be generalized to treat $\St(V,\R)$, $V$
split root graded, in a case-free manner?}

For a slightly different type of Steinberg group and a unit regular $V$,
centrality of $\Ker(\pi)$ is shown in \cite[Th.~1.12]{simp}.

\comments{ More problems in `steinprob.tex': Analogue of Matsumoto's Theorem, simply connected versions of $\PE(V)$}

%%%%%%%%%%%%%%%%%%    Biblio %%%%%%%%%%%%%%%%%%%%%%%%%%%%%%%%%%%%%

\end{document}